\documentclass[11pt, a4paper]{article}

\usepackage{amsmath,amssymb,amsfonts} 
\usepackage{amsthm}
\usepackage{amsthm}

\usepackage[latin1]{inputenc}

\theoremstyle{plain}
\newtheorem{thm}{Theorem}[section]
\newtheorem{lem}[thm]{Lemma}
\newtheorem{cor}[thm]{Corollary}
\newtheorem{prop}[thm]{Proposition}
\theoremstyle{definition}
\newtheorem{defi}[thm]{Definition}

\theoremstyle{remark}

\newtheorem{rem}[thm]{Remark}
\newtheorem{com}[thm]{Comment}

\renewenvironment{proof}{\textsc{Proof:}}{\qed}

\title{On Dynkin games with incomplete information
  }
\author{Christine Gr\"un \footnote{Laboratoire de Mathematiques de Brest
UMR 6205, 6 avenue Le Gorgeu CS 93837, 29238 BREST cedex 3, France; email: christine.gruen@univ-brest.fr.}
\footnote{Supported by the Marie Curie Initial Training Network (ITN) project: ``Deterministic and Stochastic Controlled Systems and
Application", FP7-PEOPLE-2007-1-1-ITN, No. 213841-2.}
\footnote{Ce travail a b\'en\'efici\'e une aide de l'Agence Nationale de la Recherche portant la r\'ef\'erence ANR-10-BLAN 0112}
}

\begin{document}
\maketitle

\begin{abstract}
In this paper we investigate a game of optimal stopping with incomplete information. There are two players of which only one is informed about the precise structure of the game. Observing the informed player the uninformed player is given the possibility to guess the missing information. We show that these games have a value which can be characterized as a viscosity solution to a fully non-linear variational PDE. Furthermore we derive a dual representation of the value function in terms of a minimization procedure. This representation allows under some additional assumptions to determine optimal strategies for the informed player.
\end{abstract}


\section{Introduction}

In this paper we consider a Dynkin game with incomplete information. The game starts at time $0$ and ends at time $T$ paying off a certain terminal payoff. In between the players can choose to stop the game and recieve a certain payment dependent on who stopped the game first. However with regard to  the payoffs stopping might be less favourable for them than waiting for the other one to stop the game or the game to terminate. We assume that the game is played by two players. One player is informed about the payoffs, while the other one only knows them with a certain probability $(p_i)_{i\in\{1,\ldots,I\}}$. Furthermore we assume that the players observe each other during the game so the uninformed player will try to guess his missing information.\\
Games with this kind of information incompleteness have been introduced by Aumann and Maschler (see \cite{AuMaS}) in discrete time setting. Differential games and stochastic differential games with incomplete information in their spirit have been considered in Cardaliaguet and Rainer \cite{CaRa2}, who give a characterization  of the value function in terms of a fully non linear partial differential equation. As in the case of stochastic differential games with incomplete information studied by Cardaliaguet and Rainer  \cite{CaRa2}, we allow the players to use an additional randomization device. We note that randomized stopping times have already been used in Touzi and Vieille \cite{TV} and Laraki and Solan \cite{LaSo} in a different context. As a result even if the informed player knows the exact state of nature he might not stop when it is optimal to stop for him in order to preserve his information advantage.\\
It turns out that as in the discrete time setting of Aumann and Maschler the randomization device can be interpreted as a certain minimal martingale with a state space in the probability measures on $\{1,\ldots I\}$.  With the optimal measure this representation then allows to determine optimal strategies for the informed agent. This result has been generalized to differential games by Cardaliaguet and Rainer in \cite{CaRa1} and to stochastic differential games by the author in \cite{CG}.  A similar technique of minimization over martingale measures is introduced in De Meyer \cite{DMfin} to determine optimal strategies for informed agents in a financial market.\\
In this paper we extend the previous results to the framework of Dynkin games. We show that the value function of Dynkin games with information incompleteness exists and is determined by a solution to a fully non-linear second order variational partial differential equation. We use the latter characterization in order to establish a dual representation of the value via a minimization procedure over some martingale measures. This representation then allows - under some additional assumptions - to derive optimal strategies for the informed player.\\
Dynkin games were introduced by E. Dynkin in \cite{D} as a gametheoretical version of an optimal stopping problem. Ever since there has been a vast variety of results obtained by using analytical or purely probabilistic tools. As we are considering continuous time Dynkin games with a diffusion as underlying dynamic we would notably like to mention the works of Bensoussan and Friedman \cite{BF} and Friedman \cite{Frie} who were the first to connect Dynkin games to solutions of second order variational partial differential equations. For for a probabilistic approach we refer to Alario-Nazaret, Lepeltier and Marchal \cite{ALM}, Bismut \cite{Bismut2}, Ekstr{\"o}m and Peskir \cite{Ek1}, Eckstr{\"o}m and Villeneuve \cite{Ek2}, Lepeltier and Maingueneau \cite{LeMain}, Morimoto \cite{Mori}, Stettner \cite{Stettner} and the recent work of {Kobylanski}, {Quenez}  et de Campagnolle \cite{Kob}. In combination with controlled diffusions also BSDE methods were applied by Cvitanic and Karatzas \cite{CK} and Hamad\`ene and Lepeltier \cite{HaLe2}. 
Though the extension of the current paper to Dynkin games, where also the drift of the diffusion is controlled, might seem rather straight forward there are some subtleties to consider. Especially when generalizing the BSDE approach of \cite{CG} to an approach with reflected BSDE we have to take into account that for the well-posedness of reflected BSDE as in Hamad\`ene and Lepeltier \cite{HaLe2} or Hamad\`ene and Hassani \cite{HH} one basically needs that $\bold p$ is continuous. This however implies a severe restriction on the set of martingale measures $\mathcal{P}(t,p)$, making it impossible to just follow the proofs in \cite{CG}.\\
Of course our way to consider information incompleteness is rather specific and far from being the only way to model Dynkin games with incomplete information. A very interesting paper with a completely different ansatz is the recent work of Lempa and Matom\"aki \cite{LM}.\\

\section{Description of the game}

\subsection{Canonical setup and standing assumptions}

Let $\mathcal{C}([0,T];\mathbb{R}^d)$ be the set of continuous functions from $\mathbb{R}$ to $\mathbb{R}^d$, which are constant on $(-\infty,0]$ and on $[T,+\infty)$. We denote by $B_s(\omega_B)=\omega_B(s)$ the coordinate mapping on $\mathcal{C}([0,T];\mathbb{R}^d)$ and define $\mathcal{H}=(\mathcal{H}_s)$ as the filtration generated by $s \mapsto B_s$. We denote $\mathcal{H}_{t,s}$ the $\sigma$-algebra generated by paths up to time $s$ in $\mathcal{C}([t,T];\mathbb{R}^d)$.  Furthermore we provide $\mathcal{C}([0,T];\mathbb{R}^d)$ with the Wiener measure $\mathbb{P}_0$ on $(\mathcal{H}_s)$ and we consider the respective filtration augmented by $\mathbb{P}_0$ nullsets without changing the notation.\\

In the following we investigate a two-player zero-sum differential game starting at a time $t\geq0$ with terminal time $T$.
The dynamic is given by an uncontrolled diffusion on $(\mathcal{C}([0,T];\mathbb{R}^d),$ $(\mathcal{H}_{t,s})_{s\in[t,T]}, \mathcal{H},\mathbb{P}_0)$, i.e. for $t\in[0,T], x\in\mathbb{R}^d$
\begin{eqnarray}
 dX^{t,x}_s=b(s,X^{t,x}_s)ds+a(s,X^{t,x}_s)dB_s\ \ \ \ 
 X^{t,x}_{t}=x.
\end{eqnarray}
Let $I\in\mathbb{N}^*$ and $\Delta(I)$ denote the simplex of $\mathbb{R}^I$. The objective to optimize is characterized by
\begin{itemize}
	\item[(i)] terminal payoffs: $(g_i)_{i\in\{1,\ldots, I\}}:\mathbb{R}^d\rightarrow\mathbb{R}$,
	\item[(ii)] early execution payoffs for Player 2: $(f_i)_{i\in\{1,\ldots, I\}}:[0,T]\times\mathbb{R}^d\rightarrow\mathbb{R}$,
	\item[(iii)] early execution payoffs for Player 1: $(h_i)_{i\in\{1,\ldots, I\}}:[0,T]\times\mathbb{R}^d\rightarrow\mathbb{R}$,
\end{itemize}
which are chosen with probability $p=(p_i)_{i\in\{1,\ldots, I\}}\in \Delta(I)$ before the game starts. Player 1 chooses $\tau\in[0,T]$ to minimize, Player 2 chooses $\sigma\in[0,T]$ to maximize the expected payoff:
\begin{equation}
\begin{array}{l}
	J_i(t,x,\tau,\sigma)=\mathbb{E}\bigg[f_i(\sigma,X^{t,x}_\sigma)1_{\sigma<\tau,\sigma<T}+h_i(\tau,X^{t,x}_\tau)1_{\tau\leq\sigma,\tau<T}+ g_i(X^{t,x}_T) 1_{\sigma=\tau=T}\bigg].
\end{array}
\end{equation}
We assume that both players observe their opponents control. However Player 1 knows which payoff he minimizes, Player 2 just knows the respective probabilities $p_{i}$ for scenario $i\in\{1,\ldots, I\}$.\\

The following will be the standing assumption throughout the paper.\\

{\bf Assumption (A)}
\begin{itemize}
	\item[(i)] $b:[0,T]\times\mathbb{R}^d\rightarrow \mathbb{R}^d$ is bounded and Lipschitz continuous with respect to $(t,x)$. For $1\leq k,l\leq d$ the function $\sigma_{k,l}:[0,T]\times\mathbb{R}^d \rightarrow \mathbb{R}$ is bounded and Lipschitz continuous with respect to $(t,x)$.
	
	\item[(ii)] $(g_i)_{i\in\{1,\ldots, I\}}:\mathbb{R}^d \rightarrow \mathbb{R}$, $(f_i)_{i\in\{1,\ldots, I\}}:[0,T]\times\mathbb{R}^d\rightarrow\mathbb{R}$ and $(h_i)_{i\in\{1,\ldots, I\}}:[0,T]\times\mathbb{R}^d\rightarrow\mathbb{R}$ are bounded and Lipschitz continuous. For all ${i\in\{1,\ldots, I\}}$, $t\in[0,T]$, $x\in\mathbb{R}^d$ we have that
	\begin{equation}
	  f_i(t,x)\leq h_i(t,x)
	  \end{equation}
	  and 
	  \begin{equation}
	  f_i(T,x)\leq g_i(x)\leq h_i(T,x).
	  \end{equation}
\end{itemize}

\begin{rem}
Note that (A) (ii) implies: for all $t\in[0,T]$, $x\in\mathbb{R}^d$, $p\in\Delta(I)$
\begin{equation}
	  \langle p, f(t,x)\rangle \leq \langle p,h(t,x)\rangle
\end{equation}
and
\begin{equation}
	  \langle p, f(T,x)\rangle\leq\langle p,g(x)\rangle \leq \langle p,h(T,x)\rangle.
\end{equation}
\end{rem}

\subsection{Random stopping times}

In Dynkin games both players have the possibility to stop the game with undergoing a certain punishment (early execution payment), so strategies in this case consist of a stopping decision.

\begin{defi}
At time $t\in[0,T]$ an admissible stopping time for either player for the game terminating at time $T$ is a $(\mathcal{H}_{t,s})_{s\in[{t,T}]}$ stopping time with values in $[t,T]$. We denote the set of admissible stopping times by $\mathcal{T}(t,T)$. In the following we shall omit $T$ in the notation whenever it is obvious.
\end{defi}

As in \cite{LaSo}, \cite{TV} we allow the players to choose their stopping decision randomly
\begin{defi}
A randomized stopping time after time $t\in[0,T]$ is a measurable function $\mu:[0,1]\times\mathcal{C}([t,T];\mathbb{R}^d)\rightarrow[t,T]$ such that for all $r\in[0,1]$
\[
\tau^r(\omega):=\mu(r,\omega)\in\mathcal{T}(t).
\]
We denote the set of randomized stopping times by $\mathcal{T}^r(t)$.
\end{defi}

For any $(t,x,p)\in[0,T]\times\mathbb{R}^d\times\Delta(I)$, $\mu=(\mu_i)_{i\in\{1,\ldots, I\}}\in (\mathcal{T}^r(t))^I, \nu \in \mathcal{T}^r(t)$ we set for ${i\in\{1,\ldots, I\}}$
        		\begin{equation}
		\begin{array}{rcl}
        			J_i(t,x,\mu_i,\nu)&=&\mathbb{E}_{\mathbb{P}_0\otimes\lambda\otimes\lambda}\bigg[f_i(\nu,X^{t,x}_\nu)1_{\nu<{\mu_i},\nu<T}\\
			&&\ \ \ \ \ \ \ \ \ \ \ \ \ \ \ \ \ +h_i(\mu_i,X^{t,x}_{\mu_i})1_{ \mu_i\leq\nu,\mu_i<T}+ g_i(X^{t,x}_T) 1_{\mu_i=\nu=T}\bigg],
		\end{array}
		\end{equation}
		where $\lambda$ denotes the Lebesgue measure on $[0,1]$. (In the following we will skip the subscript $ {\mathbb{P}_0\otimes\lambda\otimes\lambda}$.) Furthermore we set
		\begin{equation}
		J(t,x,p,\mu,\nu)=\sum_{i=1}^I p_i J_i(t,x,\mu_i,\nu).
		\end{equation}
		We note that the information advantage of Player 1 is reflected in (2.8) by having the possibility to choose a randomized stopping time $\mu_i$  for each state of nature $i\in\{1,\ldots,I\}$.\\
		
\subsection{An example}	
 To illustrate the importance of not immediately revealing the information advantage we would like to conclude this section with a basic deterministic example. Assume that the game takes place between times $t=0$ and $T=1$. There are two possible states of nature $i=1,2$ picked with probability $(p,1-p)$ before the game starts. They are associated to the two payoff functionals
\begin{equation}
J_1(\tau,\sigma)=(2\tau+1)1_{\tau<\sigma,\tau<1}+(2\sigma-1)1_{\sigma\leq\tau,\sigma<1}+2\ 1_{\sigma=\tau=1}
\end{equation}
and
\begin{equation}
J_2(\tau,\sigma)=(3-\tau)1_{\tau<\sigma,\tau<1}+(2-\sigma)1_{\sigma\leq\tau,\sigma<1}+\frac{3}{2}\ 1_{\sigma=\tau=1}.
\end{equation}
 Player 1, who is informed about the actual state of nature, chooses $\tau\in[0,1]$ to minimize and Player 2 chooses $\sigma\in[0,1]$ to maximize the payoff functional. However Player 2 is not informed whether it is $J_1$ or $J_2$ he has to optimize.\\
Now if the informed player plays a revealing strategy: he immediately stops the game i.e. $\tau=0$, if  $i=1$ is picked, and the payoff is $J_1(0,\sigma)=1$. In case $i=2$ he does not stop, i.e. $\tau=1$, for $i=2$. Player 2 does not know $i$ a priori, but if he sees that the revealing Player 1 does not stop he can be sure $i=2$, hence the information advantage is lost. In this case it is optimal for Player 2 to stop immediately which yields the payoff $J_2(\tau,0)=2$. So the overall payoff for a revealing strategy of Player 1 would be $p J_1(0,\sigma) + (1-p) J_2(\tau,0)=2-p$.\\
On the other hand if Player 1 plays non-revealing, that means acting as if he does not know $i$, both player face a stopping game with payoff
\begin{equation}
\begin{array}{l}
((3-2p)+(3p-1)\tau)1_{\tau<\sigma,\tau<1}+((2-3p)+(3p-1)\sigma)1_{\sigma\leq\tau,\sigma<1}\\
\ \\
\ \ \ \ \  \ \ \  \ \ \ \  \ \ \  \ \ \ \ \ \ \ \ 	+(\frac{3}{2} +\frac{1}{2}p)\ 1_{\sigma=\tau=1},
\end{array}
\end{equation}
where only $p\in[0,1]$ is known to both players. For $p<\frac{1}{7}$ the uninformed player in his turn will stop immediately. Hence in this case, we have an overall payoff of $p J_1(\tau,0) + (1-p) J_2(\tau,0)=2-3p$, which is indeed smaller than the revealing case. As we see later in section 6.3. in general a mixing of randomly revealing and non-revealing strategies will be optimal for the informed player.

\section{Value of the game}

For any $(t,x,p)\in[0,T]\times\mathbb{R}^d\times\Delta(I)$ we define  the lower value function by
\begin{equation}
					V^-(t,x,p)=\sup_{\nu \in \mathcal{T}^r(t)} \inf_{\mu\in (\mathcal{T}^r(t))^I} J(t,x,p,\mu,\nu)
\end{equation}
and the upper value function by
\begin{equation}
		 			V^+(t,x,p) = \inf_{\mu\in (\mathcal{T}^r(t))^I}\sup_{\nu \in \mathcal{T}^r(t)} J(t,x,p,\mu,\nu).
\end{equation}

\begin{rem}
It is well known (e.g. \cite{CaRa2} Lemma 3.1) that it suffices for the uninformed player to use admissible non-random strategies in (3.2). So we can use the easier expression
\begin{equation}
		 			V^+(t,x,p) = \inf_{\mu\in (\mathcal{T}^r(t))^I}\sup_{\sigma \in \mathcal{T}(t)} J(t,x,p,\mu,\sigma).
\end{equation}

\end {rem}

To show that the game has a value we establish:

\begin{thm}
       		 For any $(t,x,p)\in[0,T]\times\mathbb{R}^d\times\Delta(I)$ the value of the game is given by
		 \begin{equation}
		 			V(t,x,p):=V^+(t,x,p)=V^-(t,x,p).
		\end{equation}
\end{thm}

\begin{rem}
Note that by definition $V^+(t,x,p)\geq V^-(t,x,p)$.
\end{rem}

To establish $V^+(t,x,p)\leq V^-(t,x,p)$ we will show that $V^+$ is a viscosity subsolution and $V^-$ a viscosity supersolution to a nonlinear obstacle problem. More precisely we define the differential operator $\mathcal{L}[w](t,x,p):=\frac{1}{2}\textnormal{tr}(aa^*(t,x)D_x^2w(t,x,p))+b(t,x)D_xw(t,x,p)$ and consider
\begin{equation}
				\begin{array}{l}
	\max \bigg\{ \max\{\min\{(-\frac{\partial} {\partial t}-\mathcal{L})[w],w-\langle f(t,x),p\rangle\},\\
	 \ \ \ \ \ \ \ \ \ \ \ \ \ \ \ \ \ \ \ \ \ \ \ \ \ \ \ \ \ \ \ \ \ \ \ \ \ \ \ \ \ \ \ \ \ w-\langle h(t,x),p\rangle\},-\lambda_{\min}\left(p,\frac{\partial ^2 w}{\partial p^2}\right)\bigg\}=0
		\end{array}
			\end{equation}
with terminal condition $w(T,x,p)=\sum_{i=1,\ldots,I}p_ig_i(x)$, where for all $p\in\Delta(I)$, $A\in\mathcal{S}^I$  (where $\mathcal{S}^I$ denotes the set of symmetric $I\times I$ matrices)
\begin{eqnarray*}
\lambda_{\min}(p,A):=\min_{z\in T_{\Delta(I)(p)}\setminus\{0\}} \frac{\langle Az,z\rangle}{|z|^2}.
\end{eqnarray*}
and  $T_{\Delta(I)(p)}$ denotes the tangent cone to $\Delta(I)$ at $p$, i.e. $T_{\Delta(I)(p)}=\overline{\cup_{\lambda>0}(\Delta(I)-p)/\lambda}$.

\begin{rem}
Note that since by (2.5), (2.6)  the obstacles are separated, we one can consider as in the classical case (3.5) as
\begin{equation}
				\begin{array}{l}
	\max \bigg\{ \min\{\max\{(-\frac{\partial} {\partial t}-\mathcal{L})[w],w-\langle h(t,x),p\rangle\},\\
	 \ \ \ \ \ \ \ \ \ \ \ \ \ \ \ \ \ \ \ \ \ \ \ \ \ \ \ \ \ \ \ \ \ \ \ \ \ \ \ \ \ \ \ \ \  w-\langle f(t,x),p\rangle\},-\lambda_{\min}\left(p,\frac{\partial ^2 w}{\partial p^2}\right)\bigg\}=0.
		\end{array}
\end{equation}

\end{rem}

\begin{defi}
A function $w:[0,T]\times\mathbb{R}^d\times\Delta(I)\rightarrow\mathbb{R}$ is a viscosity subsolution to (3.5) if and only if for all $(\bar t,\bar x, \bar p)\in[0,T)\times\mathbb{R}^d\times\textnormal{Int}(\Delta(I))$ and any test function $\phi:[0,T]\times\mathbb{R}^d\times\Delta(I)\rightarrow\mathbb{R}$ such that $w-\phi$ has a (strict) maximum at $(\bar t,\bar x, \bar p)$ with $w(\bar t,\bar x, \bar p)-\phi(\bar t,\bar x, \bar p)=0$ we have, that
\begin{equation*}
				\begin{array}{l}
	\max \bigg\{ \max\{\min\{(-\frac{\partial} {\partial t}-\mathcal{L})[\phi],\phi-\langle f(t,x),p\rangle\},\\
	 \ \ \ \ \ \ \ \ \ \ \ \ \ \ \ \ \ \ \ \ \ \ \ \ \ \ \ \ \ \ \ \ \ \ \ \ \ \ \ \ \ \ \ \ \ \ \phi-\langle h(t,x),p\rangle\},-\lambda_{\min}\left(p,\frac{\partial ^2 \phi}{\partial p^2}\right)\bigg\}\leq0
		\end{array}
			\end{equation*}
at $(\bar t,\bar x,\bar p)$.
This is equivalent to:
\begin{itemize}
\item [(i)] $\lambda_{\min}\left(p,\frac{\partial ^2 \phi}{\partial p^2}\right)\geq 0$
\item[(ii)] $w(\bar t,\bar x, \bar p)=\phi(\bar t,\bar x, \bar p)\leq\langle h(\bar t, \bar x),\bar p\rangle$
\item[(iii)] If $w(\bar t,\bar x, \bar p)=\phi(\bar t,\bar x, \bar p)>\langle f(\bar t, \bar x),\bar p\rangle$, then $(\frac{\partial} {\partial t}+\mathcal{L})[\phi](\bar t,\bar x, \bar p)\geq 0$.
\end{itemize}
\end{defi}

\begin{defi}
A function $w:[0,T]\times\mathbb{R}^d\times\Delta(I)\rightarrow\mathbb{R}$ is a viscosity supersolution to (3.5) if and only if for all $(\bar t,\bar x, \bar p)\in[0,T)\times\mathbb{R}^d\times\Delta(I)$ and any test function $\phi:[0,T]\times\mathbb{R}^d\times\Delta(I)\rightarrow\mathbb{R}$ such that $w-\phi$ has a (strict) minimum at $(\bar t,\bar x, \bar p)$ with $w(\bar t,\bar x, \bar p)-\phi(\bar t,\bar x, \bar p)=0$ we have, that
\begin{equation*}
				\begin{array}{l}
	\max \bigg\{ \min\{\max\{(-\frac{\partial} {\partial t}-\mathcal{L})[\phi],\phi-\langle f(t,x),p\rangle\},\\
	 \ \ \ \ \ \ \ \ \ \ \ \ \ \ \ \ \ \ \ \ \ \ \ \ \ \ \ \ \ \ \ \ \ \ \ \ \ \ \ \ \ \ \ \ \ \ \phi-\langle h(t,x),p\rangle\},-\lambda_{\min}\left(p,\frac{\partial ^2 \phi}{\partial p^2}\right)\bigg\}\geq0
		\end{array}
			\end{equation*}
at $(\bar t,\bar x,\bar p)$.
This is equivalent to: if
\[\lambda_{\min}\left(p,\frac{\partial ^2 \phi}{\partial p^2}\right)> 0,\]
we have, that
\begin{itemize}
\item[(i)] $w(\bar t,\bar x, \bar p)=\phi(\bar t,\bar x, \bar p)\geq\langle f(\bar t, \bar x),\bar p\rangle$
\item[(ii)] If $w(\bar t,\bar x, \bar p)=\phi(\bar t,\bar x, \bar p)<\langle h(\bar t, \bar x),\bar p\rangle$, then  $(\frac{\partial} {\partial t}+\mathcal{L})[\phi](\bar t,\bar x, \bar p)\leq 0$.
\end{itemize}
\end{defi}

An essential part of the proof of Theorem 3.2. is given by the following comparison result. We postpone the proof to the appendix.

\begin{thm}
Let $w_1:[0,T] \times \mathbb{R}^d\times \Delta(I)\rightarrow\mathbb{R}$ be a bounded, continuous viscosity subsolution to (3.5), which is uniformly Lipschitz continuous in $p$, and $w_2:[0,T] \times \mathbb{R}^d\times \Delta(I)\rightarrow\mathbb{R}$ be a bounded, continuous  viscosity supersolution to (3.5), which is uniformly Lipschitz continuous in $p$. Assume that 
\begin{equation}
w_1(T,x,p)\leq w_2(T,x,p)
\end{equation}
for all $x\in\mathbb{R}^d,p\in\Delta(I)$. Then
\begin{equation}
w_1(t,x,p)\leq w_2(t,x,p)
\end{equation}
for all $(t,x,p)\in[0,T]\times\mathbb{R}^d\times\Delta(I)$.
\end{thm}

\section{Dynamic programming}

\subsection{Regularity properties}

\begin{prop}
$V^+(t,x,p)$ and $V^-(t,x,p)$ are uniformly Lipschitz continuous in $x$ and $p$ and H\"older continuous in $t$.
\end{prop}

\begin{proof}
The proof of the Lipschitz continuity in $x$ and $p$ is straightforward and omitted here. For the H\"older continuity in time let $t,t'\in[0,T]$ with $t\leq t'$. Assume $V^+(t,x,p)>V^+(t',x,p)$. Then
\begin{equation*}
\begin{array}{rcl}
0&<& V^+(t,x,p)-V^+(t',x,p)\\
\ \\
&=&\inf_{\mu\in (\mathcal{T}^r(t))^I}\sup_{\nu \in \mathcal{T}^r(t)} J(t,x,p,\mu,\nu)-\inf_{\mu\in (\mathcal{T}^r(t'))^I}\sup_{\nu \in \mathcal{T}^r(t')} J(t',x,p,\mu,\nu).
\end{array}
\end{equation*}
Now for $\epsilon>0$ choose $\bar \mu\in(\mathcal{T}^{r}(t'))^I$ $\epsilon$-optimal for $V^+(t',x,p)$. Since $t\leq t'$ we have $\bar \mu\in(\mathcal{T}^{r}(t))^I$. Furthermore choose $\bar \nu\in\mathcal{T}^r(t)$ $\epsilon$-optimal for $\sup_{\nu\in\mathcal{T}^r(t)}J(t,x,p,\bar \mu, \nu)$ and define $\hat \nu\in\mathcal{T}^r(t')$
\begin{equation}
\hat \nu=\begin{cases}
 t' & \text{on } \{ \bar \nu<t'\}\\
\bar \nu & \text{on } \{ \bar \nu\geq t'\}. \\
\end{cases}
\end{equation}
Then we have 
\begin{equation}
\begin{array}{rcl}
V^+(t,x,p)-V^+(t',x,p)-2 \epsilon \leq J(t,x,p,\bar \mu,\bar \nu)-J(t',x,p,\bar \mu,\hat \nu).
\end{array}
\end{equation}
Since
\begin{equation*}
\begin{array}{l}
J(t,x,p,\bar \mu,\bar \nu)-J(t',x,p,\bar \mu,\hat \nu)\\
\ \\
=\mathbb{E}[(f_i(\bar \nu,X^{t,x}_{\bar \nu})-f_i(t',x))1_{\bar \nu<t'}]\\
\ \\
\ \ \ +\mathbb{E}\bigg[f_i(\bar \nu,X^{t,x}_{\bar \nu})-f_i(\bar \nu,X^{t',x}_{\bar \nu})1_{t'\leq\bar \nu<{\bar \mu_i},\bar \nu<T}+h_i(\bar \mu_i,X^{t,x}_{\bar \mu_i})-h_i(\bar \mu_i,X^{t',x}_{\bar \mu_i})1_{ t'\leq\bar \mu_i\leq\bar \nu,\bar \mu_i<T}\\
\ \ \  \ \ \ \  \ \ \ \ \ \ \ \ \ \ \ \ \ \ +(g_i(X^{t,x}_T)-g_i(X^{t',x}_T)) 1_{\bar \mu_i=\bar \nu=T}\bigg],
\end{array}
\end{equation*}
the claim follows with assumption (A) by standard estimates, since $\epsilon$ can be chosen arbitrarily small. The case $V^+(t,x,p)<V^+(t',x,p)$ follows by similar arguments.

\end{proof}

The following is a key property in games with incomplete information (see \cite{AuMaS}). Our proof follows closely \cite{CaRa2}.

\begin{prop}
For all $(t,x)\in[0,T]\times\mathbb{R}^d$ $V^+(t,x,p)$ and $V^-(t,x,p)$ are convex in $p$.
\end{prop}

\begin{proof}
That $V^-(t,x,p)$ is convex in $p$ can be easily seen by the following reformulation
\begin{equation}
\begin{array}{rcl}
V^-(t,x,p)&=&\sup_{\nu \in \mathcal{T}^r(t)} \inf_{\mu\in (\mathcal{T}^r(t))^I} J(t,x,p,\mu,\nu)\\
\ \\
	&=&\sup_{\nu \in \mathcal{T}^r(t)} \sum p_i \inf_{\mu\in \mathcal{T}^r(t)} J_i(t,x,p,\mu,\nu).
\end{array}
\end{equation}
To show that $V^+(t,x,p)$ is convex in $p$: fix  $(t,x)\in[0,T]\times\mathbb{R}^d$ and let $p,p^1,p^2\in\Delta(I)$, $\lambda\in[0,1]$ such that  $p=\lambda p^1+(1-\lambda) p^2$. \\
Furthermore choose $\mu^1\in (\mathcal{T}^r(t))^I$, $\mu^2\in (\mathcal{T}^r(t))^I$  $\epsilon$-optimal for $V^+(t,x,p^1)$, $V^+(t,x,p^2)$ respectively.
Then as in \cite{CaRa2} Proposition 2.1. one can construct a $\hat{\mu}\in (\mathcal{T}^r(t))^I$, such that for any $\nu\in\mathcal{T}^r(t)$ we have that
\begin{equation}
\sum_{i=1}^I p_i J_i(t,x,\hat \mu_i,\nu)=\lambda \sum_{i=1}^I p^1_i J_i(t,x,\mu^1_i,\nu)+(1-\lambda) \sum_{i=1}^I p^2_i J_i(t,x,\mu^2_i,\nu).
\end{equation}
Maximizing over $\nu\in\mathcal{T}^r(t)$ (4.4) yields then
\begin{equation*}
 V^+(t,x,p)\leq   \lambda V^+(t,x,p^1)+ (1-\lambda )V^+(t,x,p^2)+2\epsilon
\end{equation*}
and the result follows since $\epsilon$ can be chosen arbitrarily small.
\end{proof}

Furthermore from the very definition of $V^+,V^-$ we have the following:
\begin{prop}
For all $(t,x,p)\in[0,T]\times\mathbb{R}^d\times \Delta(I)$ we have that 
\begin{equation}
 \langle f(t,x),p\rangle \leq V^+(t,x,p)\leq  \langle h(t,x),p\rangle
\end{equation}
and
\begin{equation}
  \langle f(t,x),p\rangle \leq V^-(t,x,p)\leq  \langle h(t,x),p\rangle.
\end{equation}
\end{prop}

\subsection{Subdynamic programming principle for $V^+$}

\begin{thm}
Let $(\bar t,\bar x,\bar p)\in[0,T]\times \mathbb{R}^d\times\Delta(I)$. Then for any $t\in[\bar t,T]$
\begin{equation}
\begin{array}{l}
V^+ (\bar t,\bar x,\bar p)\leq\inf_{\tau\in\mathcal{T}(\bar t,t)} \sup_{\sigma\in\mathcal{T}(\bar t,t)}\mathbb{E}\bigg[\left\langle \bar p,f(\sigma, X^{\bar t,\bar x}_{\sigma}) 1_{\sigma<\tau,\sigma<t}\right\rangle\\
 \ \ \ \ \ \ \ \ \ \ \ \ \ \ \ \ \ \ \ \ \ \ \ \ \ \ \ \ \ \ \ \ \ \ \ \ +\left\langle \bar p,h(\tau,X^{\bar t,\bar x}_{\tau} )1_{\tau\leq\sigma,\tau<t}\right\rangle+ V^+(t,X^{\bar t,\bar x}_t,\bar p) 1_{\tau=\sigma=t}\bigg].
\end{array}
\end{equation}
\end{thm}

\begin{proof}
Fix $(\bar t,\bar x,\bar p)\in[0,T]\times\mathbb{R}^d\times\Delta(I)$. Let $A^j$ be a partition of $\mathbb{R}^d$ with $diam(A^j)\leq \delta$ for a $\delta>0$. For any $j\in\mathbb{N}$, choose a $y^j\in A^j$ and $\mu^j\in(\mathcal{T}^r(t))^I$ $\epsilon$-optimal for $V^+ (t,y^j,\bar p)$.\\
Furthermore choose $\bar{\mu}\in(\mathcal{T}^r(\bar t,t))^I$ to be $\epsilon$ optimal for
\begin{equation}
\begin{array}{rcl}
&&\inf_{\mu\in(\mathcal{T}^r(\bar t))^I}\sup_{\sigma\in\mathcal{T}(\bar t)} \sum_{i=1}^I \bar p_i \mathbb{E}\bigg[f_i(X^{\bar t,\bar x}_{\nu}) 1_{\nu<\mu_i,\nu<t}+h_i(X^{\bar t,\bar x}_{ \mu_i} )1_{\mu_i\leq\nu,\mu_i<t}\\
&&\ \ \ \ \ \ \ \ \ \ \ \ \ \ \ \ \ \ \ \ \ \ \ \ \ \ \ \ \ \ \ \ \ \ \ \ \ \ \ \ \ \ \ \ \ \ \ \ \ \ \ \ \ \ \ \ \ \ \ \ \ \ \ \ \ +V^+(t,X^{\bar t,\bar x}_t, \bar p) 1_{\mu_i=\nu=t}\bigg].
\end{array}
\end{equation}
We shall build with $\bar \mu$ and  $(\mu^j)_{j\in\mathbb{N}}$ a randomized stopping time $\hat\mu\in(\mathcal{T}^r(\bar t))^I$ in the following way\\
\begin{equation}
\hat \mu=\begin{cases}
\bar \mu & \text{on } \{\bar \mu<t\}\\
(\mu^j)_{j\in\mathbb{N}}& \text{on } \{\bar \mu=t, X^{\bar t,\bar x}_t\in A^j \}. \\
\end{cases}
\end{equation}
First note that for any $\sigma \in\mathcal{T}(\bar t)$.
\begin{equation}
\begin{array}{rcl}
&&\sum_{i=1}^I \bar p_i \mathbb{E}\bigg[f_i(X^{\bar t,\bar x}_{\sigma}) 1_{\sigma<\hat \mu_i,\sigma<T}+h_i(X^{\bar t,\bar x}_{\mu_i} )1_{\hat \mu_i\leq\sigma,\hat \mu_i<T} + g_i(X^{\bar t,\bar x}_T) 1_{\hat \mu_i=\sigma=T}\bigg]\\
\ \\
&&=\sum_{i=1}^I \bar p_i \mathbb{E}\bigg[f_i(X^{\bar t,\bar x}_{\sigma}) 1_{\sigma<\hat \mu_i,\sigma<t}+h_i(X^{\bar t,\bar x}_{\hat \mu_i} )1_{\hat \mu_i\leq\sigma,\hat \mu_i<t} \bigg]\\
\ \\
&&\ \ \ \ +\sum_{i=1}^I \bar p_i \mathbb{E}\bigg[f_i(X^{\bar t,\bar x}_{\sigma}) 1_{t\leq\sigma<\hat \mu_i,\sigma<T}+h_i(X^{\bar t,\bar x}_{\hat \mu_i} )1_{t\leq\hat \mu_i\leq\sigma,\hat \mu_i<T}\\
\ \\
&&\ \ \ \ \ \ \ \ \ \ \ \ \ \ \ \ \  \ \ \ \ \ \ \ \ \ \ \ \ \ \ \ \ \ \ \ \ \  \ \ \ \ \ \ \ \ \ \ \ \ \ \ \ \ \ \ \ \ \ \ \  +g_i(X^{\bar t,\bar x}_T) 1_{\hat \mu_i=\sigma=T} \bigg],
\end{array}
\end{equation}
while by the uniform Lipschitz continuity of the coefficients by (A) and of $V^+$ by Proposition 4.1.  we have for a generic constant $c>0$
\begin{equation}
\begin{array}{rcl}
&&\sum_{i=1}^I \bar p_i \mathbb{E}\bigg[f_i(X^{\bar t,\bar x}_{\sigma}) 1_{t\leq\sigma<\hat \mu_i,\sigma<T}+h_i(X^{\bar t,\bar x}_{\mu_i} )1_{t\leq\hat\mu_i\leq\sigma,\hat \mu_i<T}+g_i(X^{\bar t,\bar x}_T) 1_{\hat \mu_i=\sigma=T} \bigg]\\
\ \\
&&\leq\sum_{j\in\mathbb{N}}\sum_{i=1}^I \bar p_i \mathbb{E}\bigg[(f_i(X^{t,y^j}_{\sigma}) 1_{t\leq\sigma<\hat \mu_i,\sigma<T}+h_i(X^{t,y^j}_{\mu_i} )1_{t\leq\hat \mu_i\leq\sigma,\hat \mu_i<T}\\
&&\ \ \ \ \ \ \ \ \ \ \ \ \ \ \ \ \ \ \ \ \ \ \ \ \ \ \ \ \ \ \ \ \ \ \ \ \ \ \ \ \ \ +g_i(X^{t,y^j}_T) 1_{\hat \mu_i=\sigma=T} ) 1_{X^{\bar t,\bar x}_t\in A^j}\bigg]+c \delta\\
\ \\
&&\leq \mathbb{E}\bigg[V^+(t,X^{\bar t,\bar x}_t, \bar p) 1_{\hat \mu_i\geq t, \sigma\geq t} \bigg]+c \delta +\epsilon.
\end{array}
\end{equation}
Hence combining (4.8) with (4.10) and (4.11) and choosing $\hat\sigma\in\mathcal{T}(\bar t)$ to be $\epsilon$-optimal for $V^+ (\bar t,\bar x,\bar p)$ (3.3) we get
\begin{equation*}
\begin{array}{l}
V^+ (\bar t,\bar x,\bar p)\\
\ \\
\leq \inf_{\mu\in(\mathcal{T}^r(\bar t, t))^I} \sup_{\sigma\in\mathcal{T}(\bar t,t)}\sum_{i=1}^I \bar p_i \mathbb{E}\bigg[f_i(X^{\bar t,\bar x}_{\sigma}) 1_{\sigma<\mu_i,\sigma<t}+h_i(X^{\bar t,\bar x}_{\mu} )1_{\mu_i\leq\sigma,\hat\mu_i<t} \\
\ \\
\  \ \ \  \ \  \ \ \  \ \ \ \  \ \ \ \ \ \ \ \ \ \ \ \ \ \ \ \  \ \ \  \ \  \ \ \  \ \ \ \  \ \ \ \ \ \ \ \ \ \ \ \ \ + V^+(t,X^{\bar t,\bar x}_t,\bar p) 1_{\mu_i=\sigma=t}\bigg]+c \delta +2 \epsilon\\
\ \\
\leq \inf_{\tau\in\mathcal{T}(\bar t,t)} \sup_{\sigma\in\mathcal{T}(\bar t,t)}\mathbb{E}\bigg[\langle \bar p,f(X^{\bar t,\bar x}_{\sigma}) 1_{\sigma<\tau,\sigma<t}+h(X^{\bar t,\bar x}_{\tau} )1_{\tau\leq\sigma,\tau<t}\rangle \\
\ \\
\  \ \ \  \ \  \ \ \  \ \ \ \  \ \ \ \ \ \ \ \ \ \ \ \ \ \ \ \  \ \ \  \ \  \ \ \  \ \ \ \  \ \ \ \ \ \ \ \ \ \ \ \ \  + V^+(t,X^{\bar t,\bar x}_t,\bar p) 1_{\tau=\sigma=t}\bigg]+c \delta +2 \epsilon.
\end{array}
\end{equation*}
The claim follows since $\epsilon$ and $\delta$ can be chosen arbitrarily small.
\end{proof}
\ \\
\ \\
In contrast to the subdynamic programming for $V^+$ a superdynamic programming principle for $V^-$ can not be derived directly. As in \cite{CaRa2} we are led to consider the convex conjugate.

\subsection{Convex conjugate of $V^-$ and implications}

For $V^-:[0,T]\times\mathbb{R}^d\times\Delta(I)\rightarrow\mathbb{R}$ we define the convex conjugate $(V^-)^*:[0,T]\times\mathbb{R}^d\times\mathbb{R}^I\rightarrow\mathbb{R}$ as
\begin{equation}
	(V^-)^*(t,x,\hat p)=\sup_{p\in\Delta(I)}\{\langle\hat p, p\rangle-V^-(t,x,p)\}.
\end{equation}
Let $\phi:[0,T]\times\mathbb{R}^d\times\Delta(I)\rightarrow \mathbb{R}$ such that $V^--\phi$ has a strict global minimum at $(\bar{t},\bar x,\bar p)\in[0,T)\times\mathbb{R}^d\times\Delta(I)$ with $V^-(\bar{t},\bar x,\bar p)-\phi(\bar{t},\bar x,\bar p)=0$ and
\begin{equation}
 \lambda_{\min}\left(p,\frac{\partial ^2 \phi}{\partial p^2}\right)>0.
\end{equation}
Then by \cite{Ca} there exists a $\delta, \eta>0$ such that for all $p\in\Delta(I)$, $(t,x)\in [\bar t,\bar t+\eta]\times B_\eta(\bar x)$
\begin{equation}
V^-(t,x,p)\geq\phi(t,x,\bar p)+\langle \frac{\partial\phi}{\partial p} (t,x,\bar p),p-\bar p\rangle+ \delta |p-\bar p|^2.
\end{equation}
Consequently, for any $\hat p\in\mathbb{R}^I$
\begin{equation}
\begin{array}{l}
(V^-)^*(t,x,\hat p)=\sup_{p\in\Delta(I)}\{\langle\hat p, p\rangle-V^-(t,x,p)\}\\
\ \\
\ \ \ \leq -\phi(t,x,\bar p)+\sup_{p\in\Delta(I)}\{\langle\hat p, p\rangle-\langle \frac{\partial\phi}{\partial p} (t,x,\bar p),p-\bar p\rangle+ \delta |p-\bar p|^2\}\\
\ \\
\ \ \ \leq -\phi(t,x,\bar p)+\langle\hat p, \bar p\rangle+\frac{1}{4\delta} | \frac{\partial\phi}{\partial p} (t,x,\bar p)-\hat p|^2,
\end{array}
\end{equation}
which implies by choosing $\hat p=\frac{\partial\phi}{\partial p} (t,x,\bar p)$
\begin{equation*}
\begin{array}{rcl}
(V^-)^*(t,x,\frac{\partial\phi}{\partial p} (t,x,\bar p))
&\leq& -\phi(t,x,\bar p)+\langle\frac{\partial\phi}{\partial p} (t,x,\bar p), \bar p\rangle
\end{array}
\end{equation*}
and for $(t,x)=(\bar t,\bar x)$ with (4.15)
\begin{equation}
\begin{array}{rcl}
(V^-)^*(\bar t,\bar x,\frac{\partial\phi}{\partial p} (\bar t,\bar x,\bar p))
&=& -V(\bar t,\bar x,\bar p)+\langle\frac{\partial\phi}{\partial p} (\bar t,\bar x,\bar p), \bar p\rangle\\
&=&-\phi(\bar t,\bar x,\bar p)+\langle\frac{\partial\phi}{\partial p} (\bar t,\bar x,\bar p), \bar p\rangle.
\end{array}
\end{equation}
Note that (4.15) and (4.16) imply in particular:
\begin{lem}
If there is a test function $\phi:[0,T]\times\mathbb{R}^d\times\Delta(I)\rightarrow \mathbb{R}$ such that $V^--\phi$ has a strict global minimum at $(\bar{t},\bar x,\bar p)\in(0,T)\times\mathbb{R}^d\times\Delta(I)$ with $V^-(\bar{t},\bar x,\bar p)-\phi(\bar{t},\bar x,\bar p)=0$ and
\begin{equation}
 \lambda_{\min}\left(p,\frac{\partial ^2 \phi}{\partial p^2}\right)>0,
\end{equation}
then  $\frac{\partial(V^-)^*}{\partial p}$ exists at $(\bar t,\bar x,\hat p)$ and is equal to $\bar p$. 
 \end{lem}

\subsection{Subdynamic programming principle for $(V^-)^*$}
Instead of a superdynamic programming principle for $V^-$ we can with regard to (4.16) show a subdynamic programming principle for $(V^-)^*$. To that end the following reformulation of $(V^-)^*$ will be useful.

\begin{prop}
For any $(t,x, \hat p)\in[0,T]\times \mathbb{R}^d\times\mathbb{R}^I$ we have that
\begin{equation}
\begin{array}{rcl}
	(V^-)^*(t,x,\hat p)
	&=&\inf_{\nu\in\mathcal{T}^r(t)}\sup_{\mu\in\mathcal{T}^r(t)} \max_{i\in\{1,\dots, I\}}\left\{\hat p_i-J_i(t,x,\mu,\nu)\right\}.
\end {array}
\end{equation}
\end{prop}

We recall:
\begin{equation*}
		\begin{array}{rcl}
        			J_i(t,x,\mu,\nu)&=&\mathbb{E}\bigg[f_i(\nu,X^{t,x}_\nu)1_{\nu<\mu,\nu<T}+h_i(\mu,X^{t,x}_\mu)1_{ \mu\leq\nu,\mu<T}+ g_i(X^{t,x}_T) 1_{\mu=\nu=T}\bigg].
		\end{array}
		\end{equation*}
		
		\begin{rem}
Again as in Remark 2.1. we can rewrite (4.18) as
\begin{equation}
\begin{array}{rcl}
	(V^-)^*(t,x,\hat p)
	&=&\inf_{\nu\in\mathcal{T}^r(t)}\sup_{\tau\in\mathcal{T}(t)} \max_{i\in\{1,\dots, I\}}\left\{\hat p_i-J_i(t,x,\tau,\nu)\right\}.
\end {array}
\end{equation}
\end{rem}

\begin{proof}
Denote $w(t,x,\hat p)$ the right hand side of (4.18). Since $V^-$ is convex in $p$ we have that $((V^-)^*)^*=V^-$. Hence it suffices to prove $w^*=V^-$.\\
First we show convexity of $w$ in $\hat p$. To that end let $\hat p, \hat p^1,\hat p^2\in\mathbb{R}^I$, $\lambda\in(0,1)$ such that $\hat p=\lambda \hat p^1+ (1-\lambda)\hat p^2$. Choose $\hat \nu^1,\hat \nu^2$ $\epsilon$-optimal for $w(t,x,\hat p^1)$, $w(t,x,\hat p^2)$ respectively. Furthermore define as in \cite{CaRa2} a $\hat \nu\in\mathcal{T}^r(t)$ such that for all $\mu\in\mathcal{T}^r(t)$
\begin{equation}
	J_i(t,x,\mu,\hat \nu)=\lambda J_i(t,x,\mu,\hat \nu^1)+(1-\lambda)J_i(t,x,\mu,\hat \nu^2).
\end{equation}
Then for all $\mu\in\mathcal{T}^r(t)$
\begin{equation*}
\begin{array}{rcl}
&&\max_{i\in\{1,\dots, I\}}\left\{\hat p_i-J_i(t,x,\mu,\hat \nu)\right\}\\
\ \\
&&=\max_{i\in\{1,\dots, I\}}\left\{\lambda(\hat p_i-J_i(t,x,\mu,\hat \nu^1))+(1-\lambda)(\hat p_i-J_i(t,x,\mu,\hat \nu^2))\right\}\\
\ \\
&&\leq \lambda \max_{i\in\{1,\dots, I\}}\left\{\hat p_i-J_i(t,x,\mu,\hat \nu^1)\right\}+(1-\lambda)\max_{i\in\{1,\dots, I\}}\left\{\hat p_i-J_i(t,x,\mu,\hat \nu^2)\right\}\\
\ \\
&&\leq \lambda w(t,x,\hat p^1)+(1-\lambda) w(t,x,\hat p^2).
\end{array}
\end{equation*}
The convexity follows then by choosing $\hat \mu$ $\epsilon$-optimal for $w(t,x,\hat p)$.\\
Next we calculate $w^*$. By definition of the convex conjugate we have
\begin{equation*}
\begin{array}{l}
w^*(t,x,p)\\
\ \\
\ \ \ =\sup_{\hat p\in\mathbb{R}^I}\left\{ \langle \hat p,p\rangle+\sup_{\nu\in\mathcal{T}^r(t)}\inf_{\mu\in\mathcal{T}^r(t)} \min_{j\in\{1,\dots, I\}}\left\{J_j(t,x,\mu,\nu)-\hat p_j\right\}\right\} \\
	\ \\
\ \ \ =\sup_{\nu\in\mathcal{T}^r(t)}\sup_{\hat p\in\mathbb{R}^I}\left\{\sum_{i=1}^I p_i  \min_{j\in\{1,\dots, I\}} \inf_{\mu\in\mathcal{T}^r(t)}\left\{ J_j(t,x,\mu,\nu)+\hat p_i-\hat p_j\right\}\right\},
\end{array}
\end{equation*}
where the supremum is attained for $\hat p_j=\inf_{\mu\in\mathcal{T}^r(t)} J_j(t,x,\mu,\nu)$.
Hence
\begin{equation*}
w^*(t,x,p)=\sup_{\nu\in\mathcal{T}^r(t)}\left\{\sum_{i=1}^I p_i \inf_{\mu\in\mathcal{T}^r(t)}J_i(t,x,\mu,\nu)\right\}=\sup_{\nu\in\mathcal{T}^r(t)}\inf_{\mu\in\mathcal({T}^r(t))^I}\sum_{i=1}^I p_i J_i(t,x,\mu,\nu).
\end{equation*}
\end{proof}

As a direct consequence of (4.18) we have:
\begin{prop}
For any $(t,x, \hat p)\in[0,T]\times \mathbb{R}^d\times\mathbb{R}^I$ we have that
\begin{equation}
\begin{array}{rcl}
	 \max_{i\in\{1,\dots, I\}}\left\{\hat p_i-h_i(x)\right\}\leq(V^-)^*(t,x,\hat p)\leq \max_{i\in\{1,\dots, I\}}\left\{\hat p_i-f_i(x)\right\}.
\end {array}
\end{equation}
\end{prop}

Furthermore we have with (4.18) as in Proposition 4.1.: 
\begin{prop}
$(V^-)^*(t,x,\hat p)$ is uniformly Lipschitz continuous in $x$ and $\hat p$ and H\"older continuous in $t$.
\end{prop}

Now we can establish a subdynamic programming principle.

\begin{thm}
Let $(\bar t,\bar x,\hat p)\in[0,T]\times \mathbb{R}^d\times\mathbb{R}^I$. Then for all $t\in[\bar t,T]$
\begin{equation}
\begin{array}{l}
(V^-)^*(\bar t,\bar x,\hat p)\\
\ \\
\ \ \ \leq \inf_{\sigma\in\mathcal{T}(\bar t,t)}\sup_{\tau\in\mathcal{T}(\bar t,t)} \mathbb{E}\bigg[ \max_{i\in\{1,\dots, I\}}\{\hat p_i-f_i(X^{\bar t,\bar x}_{\nu}) \}1_{\sigma<\tau,\sigma<t}\\
\ \\
\ \ \ \ \ \ \ \ \ \ + \max_{i\in\{1,\dots, I\}} \{\hat p_i-h_i(X^{\bar t,\bar x}_{ \tau} )\}1_{ \tau\leq\sigma,\tau<t}+(V^-)^*(t,X^{\bar t,\bar x}_t, \hat p) 1_{\tau=\sigma=t}\bigg].
\end{array}
\end{equation}
\end{thm}

\begin{proof}
Fix $(\bar t,\bar x,\hat p)\in[0,T]\times\mathbb{R}^d\times\Delta(I)$. Let $A^j$ be a partition of $\mathbb{R}^d$ with $diam(A^j)\leq \delta$ for a $\delta>0$. For any $j\in\mathbb{N}$, choose a $y^j\in A^j$ and $\nu^j\in\mathcal{T}^r(t)$ $\epsilon$-optimal for $(V^-)^* (t,y^j,\hat p)$. Furthermore fix some $\bar \sigma\in\mathcal{T}(\bar t,t)$ $\epsilon$-optimal for the right hand side of (4.22).\\
We shall build with $\bar \sigma$ and  $(\nu^j)_{j\in\mathbb{N}}$ a randomized stopping time $\hat\nu\in\mathcal{T}^r(\bar t)$ in the following way:
\begin{equation}
\hat \nu=\begin{cases}
\bar \nu & \text{on } \{\bar \nu<t\}\\
(\nu^j)_{j\in\mathbb{N}}& \text{on } \{\bar \nu=t, X^{\bar t,\bar x}_t\in A^j \}. \\
\end{cases}
\end{equation}
First note that for any $\tau\in\mathcal{T}(\bar t)$
\begin{equation}
\begin{array}{l}
\max_{i\in\{1,\dots, I\}}\left\{\hat p_i-\mathbb{E}\bigg[f_i(X^{\bar t,\bar x}_{\hat \nu}) 1_{\hat \nu<\tau,\hat \nu<T}+h_i(X^{\bar t,\bar x}_{\tau} )1_{\tau\leq\hat \nu,\tau<T} + g_i(X^{\bar t,\bar x}_T) 1_{\tau=\hat \nu=T}\bigg]\right\}\\
\ \\
= \max_{i\in\{1,\dots, I\}}\bigg\{\hat p_i-\mathbb{E}\bigg[f_i(X^{\bar t,\bar x}_{\hat \nu}) 1_{\hat \nu<\tau,\hat \nu<t}+h_i(X^{\bar t,\bar x}_{\tau} )1_{\tau\leq\hat \nu,\tau<t} \bigg]\\
\ \\
\ \ \ \ \ \ \ \ \ \ \ \ \ -\mathbb{E}\bigg[f_i(X^{\bar t,\bar x}_{\hat \nu}) 1_{t\leq\hat \nu<\tau,\hat \nu<T}+h_i(X^{\bar t,\bar x}_{\tau} )1_{t\leq\tau\leq\hat \nu,\tau<T}+g_i(X^{\bar t,\bar x}_T) 1_{\tau=\hat \nu=T} \bigg]\bigg\}\\
\ \\
\leq \max_{i\in\{1,\dots, I\}}\bigg\{\mathbb{E}\bigg[(\hat p_i-f_i(X^{\bar t,\bar x}_{\hat \nu})) 1_{\hat \nu<\tau,\hat \nu<t}+(\hat p_i-h_i(X^{\bar t,\bar x}_{\tau} ))1_{\tau\leq\hat \nu,\tau<t} \bigg]\bigg\}\\
\ \\
\ \ \ \ \ \ \ \ \ \ \ \ \ \ +\max_{i\in\{1,\dots, I\}}\bigg\{ \mathbb{E}\bigg[\hat p_i 1_{t\leq \tau,t\leq \hat \nu}-f_i(X^{\bar t,\bar x}_{\hat \nu}) 1_{t\leq\hat \nu<\hat\mu,\hat \nu<T}\\
\ \\
\ \ \ \ \ \ \ \ \ \ \ \ \ \ \ \ \ \ \ \ \ \ \ \ \ \ \ \ \ \ \ \ \ \ \ \ \ \ \ \ -h_i(X^{\bar t,\bar x}_{\tau} )1_{t\leq\tau\leq\hat \nu,\tau<T}-g_i(X^{\bar t,\bar x}_T) 1_{\tau=\hat \nu=T} \bigg]\bigg\}.
\end{array}
\end{equation}
Furthermore by the uniform Lipschitz continuity of the coefficients by (A)  we have for a generic constant $c>0$
\begin{equation*}
\begin{array}{l}
\bigg|\mathbb{E}\bigg[f_i(X^{\bar t,\bar x}_{\hat \nu}) 1_{t\leq\hat \nu<\tau,\hat \nu<T}+h_i(X^{\bar t,\bar x}_{\tau} )1_{t\leq\tau\leq\hat \nu,\tau<T}+g_i(X^{\bar t,\bar x}_T) 1_{\tau=\hat \nu=T} \bigg]\\
\ \\
\ \ -\sum_{j\in\mathbb{N}}\mathbb{E}\bigg[f_i(X^{t,y^j}_{\hat \nu}) 1_{t\leq\hat \nu<\tau,\hat \nu<T}+h_i(X^{t,y^j}_{\tilde \mu} )1_{t\leq\tau\leq\hat \nu,\tau<T}+g_i(X^{t,y^j}_T) 1_{\tau=\hat\nu=T} 1_{X^{\bar t,\bar x}_t\in A^j}\bigg]\bigg|\\
\\
\leq c \delta.
\end{array}
\end{equation*}
And  since $v\mapsto \max_{i\in\{1,\dots, I\}} v_i$ is convex, we have by taking conditional expectation, the fact that $X^{\bar t,\bar x}$ is Markovian and the choice of $\hat\nu$ in (4.24) 
\begin{equation*}
\begin{array}{rcl}
&&\max_{i\in\{1,\dots, I\}}\bigg\{\mathbb{E}\bigg[\hat p_i 1_{\tau\geq t,\hat \nu\geq t}-f_i(X^{\bar t,\bar x}_{\hat \nu}) 1_{t\leq\hat \nu<\tau,\hat \nu<T}\\
\ \\
&&\ \ \ \ \ \ \ \ \ \ \ \ \ \ \ \ \ \ \ \ \ \ \ \ \ \ \ -h_i(X^{\bar t,\bar x}_{\tau} )1_{t\leq\tau\leq\hat \nu,\tau<T}-g_i(X^{\bar t,\bar x}_T) 1_{\tau=\hat \nu=T} \bigg]\bigg\}\\
\ \\
&&\ \ \ \ \leq \sum_{j\in\mathbb{N}}\ \mathbb{E}\bigg[\max_{i\in\{1,\dots, I\}}\bigg\{\hat p_i -\mathbb{E}\bigg[f_i(X^{t,y^j}_{\nu^j}) 1_{\nu^j<\tau,\nu_j<T}+h_i(X^{t,y^j}_{\tau} )1_{\tau\leq \nu^j,\tau<T}\\
\ \\
&&\ \ \ \ \ \ \ \ \ \ \ \ \ \ \ \ \ \ +g_i(X^{t,y^j}_T) 1_{\tau=\nu^j=T}\bigg]\bigg\} 1_{X^{\bar t,\bar x}_t\in A^j}1_{\tau\geq t,\hat \nu\geq t}\bigg]+c\delta\\
\\
&&\ \ \ \ \leq \sum_{j\in\mathbb{N}}\ \mathbb{E}\bigg[(V^-)^*(t,y^j,\hat p) 1_{X^{\bar t,\bar x}_t\in A^j}1_{\tau\geq t,\hat \nu\geq t}\bigg]+c\delta+\epsilon,
\end{array}
\end{equation*}
which yields with the Lipschitz property of $(V^-)^*$ in $x$ by Proposition 4.6.
\begin{equation}
\begin{array}{rcl}
&&\max_{i\in\{1,\dots, I\}}\bigg\{\mathbb{E}\bigg[\hat p_i 1_{\tau\geq t,\hat \nu\geq t}-f_i(X^{\bar t,\bar x}_{\hat \nu}) 1_{t\leq\hat \nu<\tau,\hat \nu<T}\\
&&\ \ \ \ \ \ \ \ \ \ \ \ \ \ \ \ \ \ \ \ \ \ \ \ \ \ \ \ \ \ \ \ \ -h_i(X^{\bar t,\bar x}_{\tilde \mu} )1_{t\leq\tau\leq\hat \nu,\tau <T}-g_i(X^{\bar t,\bar x}_T) 1_{\tau=\hat \nu=T} \bigg]\bigg\}\\
\ \\
&&\ \ \ \ \leq  \mathbb{E}\bigg[(V^-)^*(t,X^{\bar t,\bar x}_t,\hat p)1_{\tau\geq t,\hat \nu\geq t}\bigg]+2c\delta+\epsilon.
\end{array}
\end{equation}
Let $\hat\tau\in\mathcal{T}(\bar t)$ be $\epsilon$-optimal for $(V^-)^* (\bar t,\bar x,\hat p)$ (4.18) 
then combining (4.24) with (4.25) we get
\begin{equation*}
\begin{array}{rcl}
&&(V^-)^* (\bar t,\bar x,\bar p)\\
\ \\
&&\leq \max_{i\in\{1,\dots, I\}}\bigg\{\mathbb{E}\bigg[(\hat p_i-f_i(X^{\bar t,\bar x}_{\hat \nu}) )1_{\hat \nu<\hat \tau,\hat \nu<t}+(\hat p_i-h_i(X^{\bar t,\bar x}_{\hat \tau} ))1_{\hat \tau\leq\hat \nu,\tau<t}\\
\ \\
&&\ \ \ \ \ \ \ \ \  \ \ \ \ \ \ \ \ \ \ \ \ \ \ \ \ \ \ \ \ \ \ +(V^-)^*(t,X^{\bar t,\bar x}_t, \hat p) 1_{\hat \tau\geq t, \hat \nu\geq t}\bigg]\bigg\}+\epsilon+2c \delta.
\ \\
&&\leq \inf_{\sigma\in\mathcal{T}(\bar t,t)}\sup_{\tau\in\mathcal{T}(\bar t,t)} \mathbb{E}\bigg[ \max_{i\in\{1,\dots, I\}}\{\hat p_i-f_i(X^{\bar t,\bar x}_{\sigma}) \}1_{\sigma<\tau,\sigma<t}\\
\ \\
&&\ \ \ \ \ \ \ \ \ \ + \max_{i\in\{1,\dots, I\}} \{\hat p_i-h_i(X^{\bar t,\bar x}_{ \tau} )\}1_{ \tau\leq\sigma,\tau<t}+(V^-)^*(t,X^{\bar t,\bar x}_t, \hat p) 1_{\sigma=\tau=t}\bigg]\\
\ \\
&&\ \ \ +2\epsilon+2c \delta.
\end{array}
\end{equation*}
The claim follows since $\epsilon$ and $\delta$ can be chosen arbitrarily small.

\end{proof}

\section{Viscosity solution property}
\subsection{Subsolution property for $V^+$}
\begin{thm}
$V^+$ is a viscosity subsolution to (3.5).
\end{thm}

\begin{proof}
Let $(\bar t,\bar x,\bar p) \in[0,T)\times\mathbb{R}^d\times \textnormal{Int}(\Delta(I))$ and $\phi:[0,T]\times\mathbb{R}^d\times\Delta(I)\rightarrow \mathbb{R}$ a test function such that $V^+-\phi$ has a strict global maximum at $(\bar{t},\bar x,\bar p)$ with $V^+(\bar{t},\bar x,\bar p)-\phi(\bar{t},\bar x,\bar p)=0$. 

Because of the convexity of $V^+$ by Proposition 4.2. and since $\bar p\in\textnormal{Int}(\Delta(I))$ we have
\begin{equation}
\lambda_{\min}\left(p,\frac{\partial ^2 \phi}{\partial p^2}\right)\geq0.
\end{equation}
So it remains to show
\begin{equation}
				\begin{array}{l}
	\max\{\min\{(-\frac{\partial} {\partial t}-\mathcal{L})(\phi),\phi-\langle f(t,x),p\rangle\}, \phi-\langle h(t,x),p\rangle\}\leq0
		\end{array}
\end{equation}
at $(\bar{t},\bar x,\bar p)$.\\
Note that by Proposition 4.3.  we already have
\begin{equation}
\phi(\bar t,\bar x,\bar p)-\langle h(\bar t,\bar x),\bar p\rangle\leq0.
\end{equation}
So it remains to show that for $V^+(\bar t,\bar x,\bar p) -\langle f(\bar t,\bar x),\bar p\rangle=\phi(\bar t,\bar x,\bar p) -\langle f(\bar t,\bar x),\bar p\rangle>0$ we have that 
$(-\frac{\partial \phi} {\partial t}-\mathcal{L})[\phi](\bar{t},\bar x,\bar p)\leq 0$, which is just a classical consequence of the subdynamic programming principle for $V^+$. Indeed if we set $\tau=t$ in the dynamic programming (4.22) we have for an $\epsilon(t-\bar t)$ optimal $\sigma_\epsilon\in\mathcal{T}(\bar t)$
 \begin{equation}
 \begin{array}{rcl}
\phi(\bar t, \bar x,\bar p)
&=&V^+ (\bar t,\bar x,\bar p)\\
\ \\
 &\leq&\mathbb{E}\bigg[\langle \bar p,f(X^{\bar t,\bar x}_{\sigma^\epsilon})\rangle 1_{\sigma^\epsilon<t}+ V^+(t,X^{\bar t,\bar x}_t,\bar p) 1_{\sigma^\epsilon=t}\bigg]-\epsilon(t-\bar t)\\
\ \\
&\leq& \mathbb{E}\bigg[\langle \bar p,f(X^{\bar t,\bar x}_{\sigma^\epsilon})\rangle 1_{\sigma^\epsilon<t}+ \phi(t,X^{\bar t,\bar x}_t,\bar p) 1_{\sigma^\epsilon=t}\bigg]-\epsilon(t-\bar t).
\end{array}
\end{equation}
If we now assume 
\begin{equation}
V^+(\bar t,\bar x,\bar p) -\langle f(\bar t,\bar x),\bar p\rangle=\phi(\bar t,\bar x,\bar p) -\langle f(\bar t,\bar x),\bar p\rangle>0
\end{equation}
and
\begin{equation}
(-\frac{\partial \phi} {\partial t}-\mathcal{L})[\phi](\bar{t},\bar x,\bar p)> 0,
\end{equation}
then there exists $h,\delta>0$ such that for all $(s,x)\in[\bar t,\bar t+h]\times B_h(\bar x)$
\begin{equation*}
\begin{array}{l}
\phi(s,x,\bar p) -\langle f(s,x),\bar p\rangle\geq \delta\ \ \ \ \ \ \ \ \ \ \ \ \ \textnormal{and}\ \ \ \ \ \ \ \ \ \ \ \ 
(-\frac{\partial \phi} {\partial t}-\mathcal{L})[\phi](s,x,\bar p)\geq \delta.
\end{array}
\end{equation*}
Define $A:=\{\inf_{s\in [\bar t, t]} |X^{\bar t,\bar x}_s-\bar x |>h\}$ and note that there exists a constant $c$ depending only on the parameters of $X^{\bar t.\bar x}$ such that $\mathbb{P}[A]\leq \frac{c(t-\bar t)^2}{h^4}$. By the It\^o formula we have since the coefficients $\phi$ and all its derivatives are bounded
\begin{equation*}
\begin{array}{rcl}
\phi(\bar t,\bar x,\bar p)&=&\mathbb{E}\left[\phi({\sigma^\epsilon},X^{\bar t,\bar x}_{{\sigma^\epsilon}},\bar p)+\int_{\bar t}^{{\sigma^\epsilon}} (-\frac{\partial}{\partial t}-\mathcal{L}) (s,X^{\bar t,\bar x}_s,\bar p) ds\right]\\
\ \\
&\geq&\mathbb{E}\left[1_{A^c}\left(\phi({\sigma^\epsilon},X^{\bar t,\bar x}_{{\sigma^\epsilon}},\bar p)+\int_{\bar t}^{{\sigma^\epsilon}} (-\frac{\partial}{\partial t}-\mathcal{L}) (s,X^{\bar t,\bar x}_s,\bar p) ds\right)\right]- c \frac{(t-\bar t)^2}{h^4}\\
\ \\
&\geq&\mathbb{E}\bigg[1_{A^c}\bigg((\langle f({\sigma^\epsilon},X^{\bar t,\bar x}_{{\sigma^\epsilon}}),\bar p\rangle+\delta) 1_{\sigma^\epsilon<t}+\phi({\sigma^\epsilon},X^{\bar t,\bar x}_{{\sigma^\epsilon}},\bar p) 1_{\sigma^\epsilon=t}\\
\ \\
&&\ \ \ \ \ \ \ \ \ \ \ \ +\delta( \sigma^\epsilon-\bar t)\bigg)\bigg]- c \frac{(t-\bar t)^2}{h^4}\\
&\geq&\mathbb{E}\bigg[\langle f({\sigma^\epsilon},X^{\bar t,\bar x}_{{\sigma^\epsilon}}),\bar p\rangle 1_{\sigma^\epsilon<t}+\phi({\sigma^\epsilon},X^{\bar t,\bar x}_{{\sigma^\epsilon}},\bar p) 1_{\sigma^\epsilon=t}\bigg]\\
\ \\
&&\ \ \ \ \ \ \ \ \ \ \ \ +\delta \mathbb{E}\left[1_{\sigma^\epsilon<t}+( \sigma^\epsilon-\bar t)\right]- 2 c \frac{(t-\bar t)^2}{h^4}.
\end{array}
\end{equation*}
Furthermore note that for $1\geq (t-\bar t)$ we have that
\begin{equation}
\mathbb{E}\left[1_{\sigma^\epsilon<t}+( \sigma^\epsilon-\bar t)\right]=\mathbb{E}\left[(1+{\sigma^\epsilon}-\bar t)1_{\sigma^\epsilon<t}+(t-\bar t)1_{\sigma^\epsilon=t}\right]\geq  (t-\bar t).
\end{equation}
So
\begin{equation*}
\begin{array}{rcl}
\phi(\bar t,\bar x,\bar p)
&\geq&\mathbb{E}\bigg[\langle f({\sigma^\epsilon},X^{\bar t,\bar x}_{{\sigma^\epsilon}}),\bar p\rangle 1_{\sigma^\epsilon<t}+\phi({\sigma^\epsilon},X^{\bar t,\bar x}_{{\sigma^\epsilon}},\bar p) 1_{\sigma^\epsilon=t}\bigg]\\
\ \\
&&\ \ \ \ \ \ \ \ \ \ \ \ +\delta (t-\bar t)- 2 c \frac{(t-\bar t)^2}{h^4},
\end{array}
\end{equation*}
which gives with (5.4)
\begin{equation*}
\delta (t-\bar t)- 2 c \frac{(t-\bar t)^2}{h^4}-\epsilon (\bar t-t)\leq 0.
\end{equation*}
Hence
\begin{equation}
\delta- 2 c \frac{(t-\bar t)}{h^4}-\epsilon\leq 0,
\end{equation}
which yields a contradiction, since $(t-\bar t)$ and $\epsilon$ can be choosen arbitrarily small.

\end{proof}

\subsection{Supersolution property of $V^-$}

With the subdynamic programming principle for $(V^-)^*$ Theorem 4.10. and the estimate in Proposition 4.9. we can now as in Theorem 5.1. establish:

\begin{thm}
$(V^-)^*$ is convex and is a viscosity subsolution to the obstacle problem
\begin{equation}
				\begin{array}{l}
	\max\bigg\{\min\left\{(-\frac{\partial } {\partial t}-\mathcal{L})[w],w-\max_{i\in\{1,\dots, I\}}\{\hat p_i-h_i(x)\}\right\},\\
	\ \ \ \ \ \ \ \ \ \ \ \ \ \ \ \ \ \ \ \ \ \ \ \ \ \ \ \ \ w-\max_{i\in\{1,\dots, I\}}\{\hat p_i-f_i(x)\}\bigg\}=0
		\end{array}
			\end{equation}
with terminal condition $w(T,x,p)=\max_{i\in\{1,\dots, I\}}\{\hat p_i-g_i(x)\}$.
\end{thm}


We are now using Theorem 4.2 to conclude the supersolution property for $V^-$.\\

\begin{thm}
$V^-$ is a viscosity supersolution to (3.5).
\end{thm}

\begin{proof}
Assume that $p=e_i$ for an $i\in\{1,\ldots,I\}$, where $e_i$ denotes the $i$-th coordinate vector in $\mathbb{R}^I$. Then (5.9) reduces to the PDE for a game with complete information, i.e.
\begin{equation}
\begin{array}{l}
	\max\{\min\{(-\frac{\partial } {\partial t}-\mathcal{L})[w],w-f_i(t,x)\},w- h_i(t,x)\}=0
\end{array}
\end{equation}
with terminal condition $w(T,x,p)=g_i(x)$ and the result is standard.\\

Let $\bar p\not\in\{e_i,\ i=1,\ldots,I\}$ and $\phi:[0,T]\times\mathbb{R}^d\times\Delta(I)\rightarrow \mathbb{R}$ such that $V^--\phi$ has a strict global minimum at $(\bar{t},\bar x,\bar p)\in [0,T)\times\mathbb{R}^d\times\Delta(I)$ with $V^-(\bar{t},\bar x,\bar p)-\phi(\bar{t},\bar x,\bar p)=0$. We have to show
 \begin{equation}
				\begin{array}{l}
	\max\big\{\max\{\min\{(-\frac{\partial } {\partial t}-\mathcal{L})[\phi],\phi-\langle f(t,x),p\rangle\},\\
	\ \\
	 \ \ \ \ \ \ \ \ \ \ \ \ \ \ \  \ \ \ \ \ \ \ \ \ \ \ \ \ \ \ \phi-\langle h(t,x),p\rangle\}, -\lambda_{\min}\left(p,\frac{\partial ^2 \phi}{\partial p^2}\right)\big\}\geq0
		\end{array}
\end{equation}
at $(\bar{t},\bar x,\bar p)$. If
\begin{equation*}
 \lambda_{\min}\left(p,\frac{\partial ^2 \phi}{\partial p^2}\right)\leq0
\end{equation*}
at $(\bar{t},\bar x,\bar p)$ (5.11) obviously holds. So assume
\begin{equation}
 \lambda_{\min}\left(p,\frac{\partial ^2 \phi}{\partial p^2}\right)>0.
\end{equation}
 Note that by Proposition 4.3. we have that $V^+(\bar t,\bar x,\bar p)-\langle f(\bar t,\bar x),p\rangle=\phi(\bar t,\bar x,\bar p)-\langle f(\bar t,\bar x),\bar p\rangle\geq 0$. So to show (5.11) it remains to show, that for $\phi(\bar t,\bar x,\bar p)<\langle h(t,x),p\rangle$, we have that
\begin{equation}
(-\frac{\partial } {\partial t}-\mathcal{L})[\phi](\bar{t},\bar x,\bar p)\geq0.
\end{equation}
Recall that (5.12) implies by Lemma 4.5. that $(V^-)^*(\bar t,\bar x,\hat p)$ is differentiable at $\hat p:=\frac{\partial\phi}{\partial p} (\bar t,\bar x,\bar p)$  with a derivative equal to $\frac{\partial (V^-)^*(\bar t,\bar x,\hat p)}{\partial \hat p}=\bar p$.\\
From Proposition 4.8. we have
\begin{equation}
\begin{array}{rcl}
(V^-)^*(\bar t,\bar x,\hat p)
&\geq&\max_{i\in\{1,\ldots,I\}}\{\hat p_i-h_i(\bar t,\bar x)\}.
\end{array}
\end{equation}
Indeed we have strict inequality in (5.14) for  $\bar p\not\in\{e_i,\ i=1,\ldots,I\}$. Assume that
\begin{equation}
\begin{array}{rcl}
(V^-)^*(\bar t,\bar x,\hat p)
&=&\max_{i\in\{1,\ldots,I\}}\{\hat p_i-h_i(\bar t,\bar x)\}.
\end{array}
\end{equation}
Since $\max_{i\in\{1,\ldots,I\}}\{\hat p_i-h_i(\bar t,\bar x)\}$ is convex in $\hat p$, we would have that $\max_{i\in\{1,\ldots,I\}}\{\hat p_i-h_i(\bar t,\bar x)\}$ is also differentiable at $\hat p$ with a derivative equal to $\frac{\partial (V^-)^*(\bar t,\bar x,\hat p)}{\partial \hat p}=\bar p$.\\
However the map $\hat p'\to \max_{i\in\{1,\ldots,I\}}\{\hat p_i'-h_i(\bar t,\bar x)\}$ is only differentiable at points for which there is a unique $i_0\in \{1, \dots, I\}$ such that $\max_{i\in\{1,\ldots,I\}}\{\hat p_i'-h_i(\bar t,\bar x)\}= \hat p_{i_0}'-h_{i_0}(\bar t,\bar x)$ and in this case its derivative is given by $e_{i_0}$. This is impossible since $\bar p\neq e_{i_0}$. Therefore
\begin{equation}
\begin{array}{rcl}
(V^-)^*(\bar t,\bar x,\hat p)
&> &\max_{i\in\{1,\ldots,I\}}\{\hat p_i-h_i(\bar t,\bar x)\}
\end{array}
\end{equation}
holds, which implies with (4.16)
\begin{equation}
\begin{array}{rcl}
V^-(\bar t,\bar x,\bar p)&<&\langle \hat p,\bar p\rangle-\max_{i\in\{1,\ldots,I\}}\{\hat p_i-h_i(\bar t,\bar x)\}\\
\ \\
&=&\langle \hat p,\bar p\rangle+\min_{i\in\{1,\ldots,I\}}\{-\hat p_i+h_i(\bar t,\bar x)\}\\
\ \\
&\leq&\langle h(\bar t,\bar x),\bar p\rangle.
\end{array}
\end{equation}
If we now recall the dynamic programming for $(V^-)^*$ with setting $\sigma=t$, i.e.
\begin{equation}
\begin{array}{rcl}
&&(V^-)^*(\bar t,\bar x,\hat p) \leq \sup_{\tau\in\mathcal{T}(\bar t,t)} \mathbb{E}\bigg[\max_{i\in\{1,\dots, I\}} \{\hat p_i-h_i(X^{\bar t,\bar x}_{ \tau} )\}1_{ \tau<t}\\
&&\ \ \ \ \ \ \ \ \ \ \ \ \ \ \ \ \ \ \ \ +(V^-)^*(t,X^{\bar t,\bar x}_t, \hat p) 1_{\tau=t}\bigg],
\end{array}
\end{equation}
we have with the upper bound of $(V^-)^*$ (5.16) that $(V^-)^*$ has the viscosity subsolution property to
\begin{equation}
(-\frac{\partial } {\partial t}-\mathcal{L})[w]=0
\end{equation}
at $(\bar t,\bar x,\hat p)$. And as in \cite{Ca} $V^-$ has the viscosity supersolution property to (5.19) at $(\bar t,\bar x,\bar p)$, hence (5.13) holds.
\end{proof}

\subsection{Viscosity solution property of the value function}

To establish Theorem 3.2. it remains with Remark 3.3. to show that $V^-\geq V^+$. This is however a direct consequence of Theorem 5.1. and Theorem 5.2. together with the comparison Theorem 3.7.. We then have the following characterization of the value.
\begin{cor}
The value function $V:[0,T]\times \mathbb{R}^d\times\Delta(I)\rightarrow\mathbb{R}$ is the unique viscosity solution to (3.5) in the class of bounded, uniformly continuous functions, which are uniformly Lipschitz continuous in $p$.
\end{cor}

\section{Alternative representation}

In a second part we use the PDE characterization to establish a representation of the value function via a minimization procedure over certain martingale measures. To do so we  enlarge the canonical Wiener space to a space which carries besides a Brownain motion $B$ a new dynamic $\bold p$. We use this additional dynamic to model the incorporation of the private information into the game. More precisely we model the probability in which scenario the game is played in according to the information of the uninformed Player 2.

\subsection{Enlargement of the canonical space}

To that end let us denote by $\mathcal{D}([0,T];\Delta(I))$ the set of c\`adl\`ag functions from $\mathbb{R}$ to $\Delta(I)$, which are constant on $(-\infty,0)$ and on $[T , +\infty)$. We denote by $\bold p_s(\omega_p)=\omega_{{p}}(s)$ the coordinate mapping on $\mathcal{D}([0,T];\Delta(I))$ and by $\mathcal{G}=(\mathcal{G}_s)$ the filtration generated by $s \mapsto \bold p_s$.
Furthermore we recall that $\mathcal{C}([0,T];\mathbb{R}^d)$ denotes the set of continuous functions from $\mathbb{R}$ to $\mathbb{R}^d$, which are constant on $(-\infty,0]$ and on $[T,+\infty)$. We denote by $B_s(\omega_B)=\omega_B(s)$ the coordinate mapping on $\mathcal{C}([0,T];\mathbb{R}^d)$ and by $\mathcal{H}=(\mathcal{H}_s)$ the filtration generated by $s \mapsto B_s$.
We equip the product space $\Omega:=\mathcal{D}([0,T];\Delta(I))\times \mathcal{C}([0,T];\mathbb{R}^d)$ with the right-continuous filtration $\mathcal{F}$, where $\mathcal{F}_t=\cap_{s>t} \mathcal{F}^0_t$ with $(\mathcal{F}^0_s)=(\mathcal{G}_s)\otimes(\mathcal{H}_s)$.  In the following we shall, whenever we work under a fixed probability $\mathbb{P}$ on $\Omega$, complete the filtration $\mathcal{F}$ with $\mathbb{P}$-nullsets without changing the notation.\\

For $0\leq t\leq T$ we denote $\Omega_t=\mathcal{D}([t,T];\Delta(I))\times\mathcal{C}([t,T];\mathbb{R}^d)$ and $\mathcal{F}_{t,s}$ the (right-continuous) $\sigma$-algebra generated by paths up to time $s\geq t$ in $\Omega_t$. Furthermore we define the space 
\[\Omega_{t,s}=\mathcal{D}([t,s];\Delta(I))\times\mathcal{C}([t,s];\mathbb{R}^d)\]
for $0\leq t\leq s \leq T$. If $r\in(t,T]$ and $\omega\in\Omega_t$ then let
\begin{eqnarray*}
\omega_1=1_{[-\infty,r)}\omega\ \ \ \ \ \ \ \ \ \ 
\omega_2=1_{[r,+\infty]}(\omega-\omega_{r-})
\end{eqnarray*}
and denote $\pi\omega=(\omega_1,\omega_2)$. The map $\pi:\Omega_t\rightarrow\Omega_{t,r}\times\Omega_{r}$ induces the identification
	$\Omega_t=\Omega_{t,r}\times\Omega_{r}$
moreover $\omega=\pi^{-1}(\omega_1,\omega_2)$, where the inverse is defined in an evident way.\\

For any measure $\mathbb{P}$ on $\Omega$, we denote by $\mathbb{E}_\mathbb{P}[\cdot]$ the expectation with respect to $\mathbb{P}$. We equip $\Omega$ with a certain class of measures.

\begin{defi}
Given $p\in\Delta(I )$, $t\in[0,T]$, we denote by $\mathcal{P}(t,p)$ the set of probability measures $\mathbb{P}$ on $\Omega$
such that, under $\mathbb{P}$
\begin{itemize}
\item[(i)]	$\bold p$ is a martingale, such that
	$\bold p_s = p$  $\forall s< t$,
	$\bold p_s \in\{e_i , i = 1, \ldots , I \}$  $\forall s \geq T$   $\mathbb{P}$-a.s., where $e_i$ denotes the $i$-th coordinate vector in $\mathbb{R}^I$, and
 $\bold p_T$ is independent of $(B_s)_{s\in(-\infty,T]},$
 \item[(ii)]	$(B_s)_{s \in [0,T]}$ is a Brownian motion. 
\end{itemize}
\end{defi}

\begin{com}
Assumption (ii) is naturally given by the Brownian structure of the game. Assumption (i) is motivated as follows. Before the game starts the information of the uninformed player is just the initial distribution $p$. The martingale property, implying $\bold p_t=\mathbb{E}_{\mathbb{P}}[\bold p_T|\mathcal{F}_t]$, is due to the best guess of the uninformed player about the scenario he is in. Finally, at the end of the game the information is revealed hence $\bold p_T \in\{e_i , i = 1, \ldots , I \}$ and  since the scenario is picked before the game starts the outcome $\bold p_T $ is independent of the Brownian motion. 
\end{com}

\subsection{Auxiliary games and representation}

From now on we will consider stopping times on the enlarged space  $\Omega=\mathcal{D}([0,T];\Delta(I))\times\mathcal{C}([0,T];\mathbb{R}^d)$ .

\begin{defi}
At time $t\in[0,T]$ an admissible stopping time for either player is a $(\mathcal{F}_{s})_{s\in[t,T]}$ stopping time with values in $[t,T]$. We denote the set of admissible stopping times by $\bar{\mathcal{T}}(t,T)$. In the following we shall omit $T$ in the notation whenever it is obvious.
\end{defi}
We note that in contrast to Definition 2.2. the admissible stopping times at time $t$ might now also depend on the paths of the Brownian motion before time $t$.\\

One can now consider a stopping game with this additional dynamic, namely with a payoff given by
\begin{equation}
\begin{array}{rcl}
J(t,x,\tau,\sigma,\mathbb{P})_{t-}&:=&\mathbb{E}_\mathbb{P}\bigg[\langle \bold p_\sigma,f(\sigma,X^{t,x}_\sigma)\rangle1_{\sigma<\tau,\sigma<T}+\langle \bold p_\tau, h(\tau,X^{t,x}_\tau)\rangle1_{ \tau\leq\sigma,\tau<T}\\
&&\ \ \ \ \ \ \ \ \ \ \ \ \ \ \ \ \ \ \ \ \ \ \ \ \ \ \ \ \ + \langle \bold p_T,g(X^{t,x}_T)\rangle 1_{\sigma=\tau=T}|\mathcal{F}_{t-}\bigg],
\end{array}
\end{equation}
 where $\tau\in\bar{\mathcal{T}}(t)$ denotes the stopping time choosen by Player 1, who minimizes, and  $\sigma\in\bar{\mathcal{T}}(t)$ denotes the stopping time choosen by Player 2, who maximizes the expected outcome. In contrast to the previous consideration here we are only working with non randomized stopping times. Indeed the randomization is in some sense shifted to the additional dynamic $\bold p$.\\
 
 Note that the known results in literature do not imply that these games have a value for any fixed $\mathbb{P}\in\mathcal{P}(t,p)$, i.e.
 \begin{equation}
 \begin{array}{rcl}
&&\textnormal{esssup}_{\sigma \in \bar{\mathcal{T}}(t)} \textnormal{essinf}_{\tau\in \bar{\mathcal{T}}(t)}J(t,x,\tau,\sigma,\mathbb{P})_{t-}\\
\ \\
&&\ \ \ \ \ \ \ \ \ \ \ =  \textnormal{essinf}_{\tau\in \bar{\mathcal{T}}(t)}\textnormal{esssup}_{\sigma \in \bar{\mathcal{T}}(t)}J(t,x,\tau,\sigma,\mathbb{P})_{t-}.
 \end{array}
\end{equation}
  Indeed since $\bold p$ is only assumed to be c\`adl\`ag the theorems of \cite{HaLe2} or \cite{HH} requiring basically the continuity of $\bold p$ do not apply. For us however it is for now not important since our first goal is an alternative representation of the value function, for which we have a PDE representation. Since $\bold p$ can be interpreted as a manipulation of the uninformed player by the informed one the outcome of the game should be some minimum in this manipulation.\\
 
 Fix $t\in[0,T]$, $x\in\mathbb{R}^d$, $p\in\Delta(I)$.  Note that all $\mathbb{P}\in \mathcal{P}(t,p)$ are equal on $\mathcal{F}_{t-}$, i.e. the distribution of $(B_s,\bold p_s)$ on $[0,t)$ is given by $\delta(p)\otimes\mathbb{P}_0$, where $\delta(p)$ is the measure under which $\bold p$ is constant and equal to $p$ and $\mathbb{P}_0$ is the Wiener measure on $\Omega_{0,t}$. So we can identify each $\mathbb{P}\in \mathcal{P}(t,p)$ on $\mathcal{F}_{t-}$ with a common probability measure $\mathbb{Q}$ and define $\mathbb{Q}\textnormal{-a.s.}$ the lower value function
\begin{equation}
\begin{array}{l}
W^-(t,x,p)=\textnormal{essinf}_{\mathbb{P}\in\mathcal{P}(t,p)} \textnormal{esssup}_{\sigma \in \bar{\mathcal{T}}(t)} \textnormal{essinf}_{\tau\in \bar{\mathcal{T}}(t)}J(t,x,\tau,\sigma,\mathbb{P})_{t-}
\end{array}
\end{equation}
and the upper value function
\begin{equation}
\begin{array}{l}
		 			W^+(t,x,p) =\textnormal{essinf}_{\mathbb{P}\in\mathcal{P}(t,p)}  \textnormal{essinf}_{\tau\in \bar{\mathcal{T}}(t)}\textnormal{esssup}_{\sigma \in \bar{\mathcal{T}}(t)}J(t,x,\tau,\sigma,\mathbb{P})_{t-},
\end{array}
\end{equation}
where by definition we have $W^-(t,x,p)\leq W^+(t,x,p)$.

\begin{thm}
 For any $(t,x,p)\in[0,T]\times\mathbb{R}^d\times\Delta(I)$ we have that
		  \begin{equation}
		 			W(t,x,p):=W^+(t,x,p)=W^-(t,x,p).
		\end{equation}
		 Furthermore the value of the Dynkin game with incomplete information can be written as
		 \begin{equation}
		 			V(t,x,p)=W(t,x,p).
		\end{equation}
\end{thm}

To prove the theorem we establish a subdynamic programming for $W^+$ and a superdynamic programming principle for $W^-$. Then we show that $W^+$ is a subsolution and $W^-$ a supersolution to the PDE (3.5). After establishing that $W^+$ and $W^-$ are bounded, uniformly continuous functions, which are uniformly Lipschitz continuous in $p$, the comparison result Theorem 3.7. gives us the equalities (6.5) and (6.6).

\subsection{Optimal strategies for the informed player}

The motivation for the alternative representation is that, as in \cite{CaRa1}, \cite{CG}  it allows to determine optimal strategies for the informed player. Indeed, if we assume that there exists a $\bar{\mathbb{P}}\in\mathcal{P}(t,p)$, such that  
\begin{equation}
\begin{array}{rcl}
V(t,x,p)&=& \textnormal{essinf}_{\tau\in \bar{\mathcal{T}}(t)}\textnormal{esssup}_{\sigma \in \bar{\mathcal{T}}(t)}J(t,x,\tau,\sigma,\bar{\mathbb{P}})_{t-},
\end{array}
\end{equation}
then we can define for any scenario $i\in\{1,\ldots,I\}$ a probability measure $\bar{\mathbb{P}}_i$ by: 
for all $A\in\mathcal{F}$ we have that
 \[\bar{\mathbb{P}}_i[A]= \bar{\mathbb{P}}[A|\bold{p}_T=e_i]=\frac{1}{p_i}  \bar{\mathbb{P}} [A\cap\{\bold{p}_T=e_i\}],\ \ \textnormal {if } p_i>0,\] 
and $\bar{\mathbb{P}}_i[A]= \bar{\mathbb{P}}[A]$ else. It is clear by Definition 6.1. that $B$ is still a Brownian motion under $\mathbb{P}^i$.\\
We note that the right-continuity of $\bold p$ allows to define the stopping time $\tau^*=\inf\{s\in[0,T], (s,X^{t,x}_s,\bold p_s)\in D\}$, where $D=\{(t,x,p)\in[0,T]\times\mathbb{R}^d\times\Delta(I): V(t,x,p)\geq\langle h(t,x),p\rangle\}$ is a closed set by the continuity of $V$ and $g$.\\
The couple $(\tau^*,\bar{\mathbb{P}}_i)$ then defines a randomized stopping time for the first player. Indeed, for each state of nature $i\in\{1,\ldots, I\}$ the informed player stops when $(s,X^{t,x}_s,\bold p_s)$ enters $D$ under $\bar{\mathbb{P}}_i$, where $X^{t,x}$ is the diffusion both players observe and $\bold p$ under $\bar{\mathbb{P}}_i$ represents his own randomization device.
 
\begin{thm}
For any scenario $i=1,\ldots,I$ and any stopping time of the uninformed player $\sigma\in\bar{\mathcal{T}}(t)$ playing  $(\tau^*,\bar{\mathbb{P}}_i)$ is optimal for the informed player in the sense that 
\begin{equation}
\begin{array}{l}
\sum_{i=1}^I p_i\mathbb{E}_{\bar{ \mathbb{P}}_i} \bigg[f_i(\sigma,X^{t,x}_\sigma)1_{\sigma<\tau^*,\sigma<T}\\
			\ \ \ \ \ \ \ \ \ \ \ \ \ \ \ \ \ +h_i(\tau^*,X^{t,x}_{\tau^*})1_{ \tau^*\leq\sigma,\tau^*<T}+ g_i(X^{t,x}_T) 1_{\tau^*=\sigma=T}\bigg]\leq V(t,x,p).
\end{array}
\end{equation}
\end{thm}

\begin{proof}
By definition of $\bar{\mathbb{P}}_i$ we have
\begin{eqnarray*}
&&\sum_{i=1}^I p_i\mathbb{E}_{\bar{ \mathbb{P}}_i} \bigg[f_i(\sigma,X^{t,x}_\sigma)1_{\sigma<\tau^*,\sigma<T}+h_i(\tau^*,X^{t,x}_{\tau^*})1_{ \tau^*\leq\sigma,\tau^*<T}+ g_i(X^{t,x}_T) 1_{\tau^*=\sigma=T}\bigg]\\
&&\ \ \ {=}\sum_{i=1}^I \bar {\mathbb{P}}[\bold {p}_T=e_i] \mathbb{E}_{\bar{ \mathbb{P}}} \bigg[f_i(\sigma,X^{t,x}_\sigma)1_{\sigma<\tau^*,\sigma<T}\\
&&			\ \ \ \ \ \ \ \ \ \ \ \ \ \ \ \ \ +h_i(\tau^*,X^{t,x}_{\tau^*})1_{ \tau^*\leq\sigma,\tau^*<T}+ g_i(X^{t,x}_T) 1_{\tau^*=\sigma=T}| \bold{p}_T=e_i \bigg]\\
&&\ \ \ {=}\sum_{i=1}^I  \mathbb{E}_{\bar{ \mathbb{P}}} \bigg[ 1_{\{\bold {p}_T=e_i\}} \bigg(f_i(\sigma,X^{t,x}_\sigma)1_{\sigma<\tau^*,\sigma<T}\\
&&			\ \ \ \ \ \ \ \ \ \ \ \ \ \ \ \ \ +h_i(\tau^*,X^{t,x}_{\tau^*})1_{ \tau^*\leq\sigma,\tau^*<T}+ g_i(X^{t,x}_T) 1_{\tau^*=\sigma=T}\bigg) \bigg]\\
&&\ \ \ {=} \mathbb{E}_{\bar{ \mathbb{P}}} \bigg[ \langle \bold p_T,f(\sigma,X^{t,x}_\sigma)\rangle 1_{\sigma<\tau^*,\sigma<T}\\
&&			\ \ \ \ \ \ \ \ \ \ \ \ \ \ \ \ \ +\langle \bold p_T,h(\tau^*,X^{t,x}_{\tau^*})\rangle 1_{ \tau^*\leq\sigma,\tau^*<T}+ \langle \bold p_T, g(X^{t,x}_T)\rangle 1_{\tau^*=\sigma=T}\bigg],
\end{eqnarray*}
while, since $\bold p$ is a martingale, we have by conditioning
\begin{eqnarray*}
&&  \mathbb{E}_{\bar{ \mathbb{P}}} \bigg[ \langle \bold p_T,f(\sigma,X^{t,x}_\sigma)\rangle 1_{\sigma<\tau^*,\sigma<T}\\
&&			\ \ \ \ \ \ \ \ \ \ \ \ \ \ \ \ \ +\langle  \bold p_T,h(\tau^*,X^{t,x}_{\tau^*})\rangle 1_{ \tau^*\leq\sigma,\tau^*<T}+\langle \bold p_T, g(X^{t,x}_T)\rangle 1_{\tau^*=\sigma=T} \bigg]\\
&&\ \ \ =\mathbb{E}_{\bar{ \mathbb{P}}} \bigg[ \langle \bold p_\sigma,f(\sigma,X^{t,x}_\sigma)\rangle 1_{\sigma<\tau^*,\sigma<T}\\
&&\ \ \ \ \ \ \ \ \ \ \ \ \ \ \ \ \ +\langle \bold p_{\tau^*},h(\tau^*,X^{t,x}_{\tau^*})\rangle 1_{ \tau^*\leq\sigma,\tau^*<T}+\langle \bold p_T, g(X^{t,x}_T)\rangle 1_{\tau^*=\sigma=T}\bigg].
\end{eqnarray*}
(6.8) follows then with  (6.7) by standard results.
\end{proof}

\subsection{The functions $W^+,W^-$ and $\epsilon$-optimal martingale measures}

We conclude this section with some important technical remarks. Note that by its very definition $W^+(t,x,p)$ and $W^-(t,x,p)$ are merely ${\mathcal{F}_{t-}}$ measurable random fields. However we can show that they are deterministic and hence a good candidate to represent the deterministic value function $V(t,x,p)$. The proof is mainly based on the methods in \cite{BuLi} using perturbation of $\mathcal{C}([0,T];\mathbb{R}^d)$ with certain elements of the Cameron-Martin space. We already adapted these arguments to the framework of games with incomplete information in \cite{CG}. The proof is very similar here and thus omitted.

\begin{prop}
For any $t\in[0,T]$, $x\in\mathbb{R}^d$, $p\in\Delta(I)$ we have that
\begin{eqnarray*}
W^+(t,x,p)&=&\mathbb{E}_{\mathbb{Q}}[W^+(t,x,p)]\ \ \ \ \ \ \ \ {\mathbb{Q}}\textnormal{-a.s.}\\
W^-(t,x,p)&=&\mathbb{E}_{\mathbb{Q}}[W^-(t,x,p)]\ \ \ \ \ \ \ \ {\mathbb{Q}}\textnormal{-a.s.}
\end{eqnarray*}
Hence identifying $W^+$, $W^-$ respectively with its deterministic version we can consider $W^+:[0,T]\times\mathbb{R}^d\times\Delta(I)\rightarrow\mathbb{R}$ and $W^-:[0,T)\times\mathbb{R}^d\times\Delta(I)\rightarrow\mathbb{R}$ as deterministic functions.
\end{prop}

In the following section we establish some regularity results and a dynamic programming principle. To this end we work with $\epsilon$-optimal measures. Note that since we are taking the essential infimum over a family of random variables, existence of an $\epsilon$-optimal $\mathbb{P}^\epsilon\in \mathcal{P}(t,p)$ is as in \cite{CG} not standard. Therefore we provide a technical lemma, the proof of which can be provided along the lines of \cite{BuLi}, \cite{CG} respectively.

\begin{lem}
For any $(t,x,p)\in[0,T]\times\mathbb{R}^d\times\Delta(I)$ there is an $\epsilon$-optimal $\mathbb{P}^\epsilon\in \mathcal{P}(t,p)$ in the sense that $\mathbb{Q}\textnormal{-a.s.}
$
\begin{equation*}
\begin{array}{rcl}
W^-(t,x,p)+\epsilon&\geq&  \textnormal{esssup}_{\sigma \in \bar{\mathcal{T}}(t)} \textnormal{essinf}_{\tau\in \bar{\mathcal{T}}(t)}J(t,x,\tau,\sigma,\mathbb{P}^\epsilon)_{t-}.
 \end{array}
 \end{equation*}
 Furthermore for any $(t,x,p)\in[0,T]\times\mathbb{R}^d\times\Delta(I)$ there is an $\epsilon$-optimal $\mathbb{P}^\epsilon\in \mathcal{P}(t,p)$ in the sense that $\mathbb{Q}\textnormal{-a.s.}
$
\begin{equation*}
\begin{array}{rcl}
W^+(t,x,p)+\epsilon&\geq&   \textnormal{essinf}_{\tau\in \bar{\mathcal{T}}(t)} \textnormal{esssup}_{\sigma \in \bar{\mathcal{T}}(t)}J(t,x,\tau,\sigma,\mathbb{P}^\epsilon)_{t-}.
 \end{array}
 \end{equation*}
\end{lem}

For technical reasons we furthermore introduce the set $\mathcal{P}^{f}(t,p)$ as the set of all measures $\mathbb{P} \in  \mathcal{P}(t,p)$, such that there exists a finite set $S\subset\Delta (I)$ with $\bold p_s\in S$ $\mathbb{P}$-a.s. for all $s\in[t,T]$.

\begin{rem}  
Note that for any $(t,x,p)\in[0,T]\times\mathbb{R}^d\times\Delta(I)$ $\epsilon>0$ we can choose an $\epsilon$-optimal $\mathbb{P}^\epsilon$ in the smaller class $\mathcal{P}^f(t,p)$.
The idea of the proof is as follows: first choose $\frac{\epsilon}{2}$-optimal measure $\mathbb{P}^\epsilon\in\mathcal{P}(t,p)$ for $W^-(t,x,p)$. Since $\bold p$ progressively measurable we can approximate it  by an elementary processes $\bar{\bold p}^\epsilon$, such that one has
\begin{equation*}
				|\textnormal{esssup}_{\sigma \in \bar{\mathcal{T}}(t)}  \textnormal{essinf}_{\tau\in \bar{\mathcal{T}}(t)}J(t,x,\tau,\sigma,\mathbb{P}^\epsilon)_{t-}- \textnormal{esssup}_{\sigma \in \bar{\mathcal{T}}(t)} \textnormal{essinf}_{\tau\in \bar{\mathcal{T}}(t)}J(t,x,\tau,\sigma,\bar{\mathbb{P}}^\epsilon)_{t-}|\leq\frac{\epsilon}{2},
\end{equation*}
where $\bar{\mathbb{P}}^\epsilon$ distribution of $(B,\bar{\bold p}^\epsilon)$. The same argument works for $W^+$.
\end{rem}

\section{Dynamic programming for $W^+,W^-$}

\subsection{Regularity properties}

\begin{prop}
For all $(t,x)\in[0,T]\times\mathbb{R}^d$ $W^+(t,x,p)$ and $W^-(t,x,p)$ are convex in $p$.
\end{prop}

\begin{proof}
Let $(t,x)\in[0,T]\times\mathbb{R}^d$ and $p_1,p_2\in\Delta(I)$. Let $\mathbb{P}^{1}\in\mathcal{P}(t,p_1)$, $\mathbb{P}^{2}\in\mathcal{P}(t,p_2)$ be $\epsilon$-optimal for $W^+(t,x,p_1)$, $W^+(t,x,p_2)$ respectively. For $\lambda\in[0,1]$ define a martingale measure $\mathbb{P}^\lambda\in\mathcal{P}(t,p_\lambda)$, such that for all measurable $\phi:\mathcal{D}([0,T];\Delta(I))\times\mathcal{C}([0,T];\mathbb{R}^d)\rightarrow\mathbb{R}_+$
\begin{eqnarray*}
\mathbb{E}_{\mathbb{P}^\lambda}[\phi(\bold p,B)]=\lambda \mathbb{E}_{\mathbb{P}^1}[\phi(\bold p,B)]+(1-\lambda)\mathbb{E}_{\mathbb{P}^2}[\phi(\bold p,B)].
\end{eqnarray*}
Observe that this can be understood as identifying $\Omega$ with $\Omega\times\{1,2\}$ with weights $\lambda$ and $(1-\lambda)$ for $\Omega\times\{1\}$ and  $\Omega\times\{2\}$, respectively. So
\begin{eqnarray*}
W^+(t,x,p_\lambda)&\leq&\textnormal{essinf}_{\tau\in \bar{\mathcal{T}}(t)}\textnormal{esssup}_{\sigma \in \bar{\mathcal{T}}(t)}  J(t,x,\tau,\sigma,\mathbb{P}^\lambda)_{t-}\\
&=& 1_{\Omega\times\{1\}} \textnormal{essinf}_{\tau\in \bar{\mathcal{T}}(t)}\textnormal{esssup}_{\sigma \in \bar{\mathcal{T}}(t)} J(t,x,\tau,\sigma,\mathbb{P}^1)_{t-}\\
&&\ \ \ \ \ \ \ \ \ \ \ \ \ +1_{\Omega\times\{2\}} \textnormal{essinf}_{\tau\in \bar{\mathcal{T}}(t)}\textnormal{esssup}_{\sigma \in \bar{\mathcal{T}}(t)} J(t,x,\tau,\sigma,\mathbb{P}^2)_{t-}\\
&\leq&  1_{\Omega\times\{1\}} W^+(t,x,p_1)+1_{\Omega\times\{2\}} W^+(t,x,p_2)+2\epsilon
\end{eqnarray*}
and the convexity follows by taking expectation, since $\epsilon$ can be chosen arbitrarily small. The proof for $W^-$ follows by similar arguments.
\end{proof}

\begin{prop}
$W^+(t,x,p)$ and $W^-(t,x,p)$ are uniformly Lipschitz continuous in $x$ and $p$ and H\"older continuous in $t$.
\end{prop}

\begin{proof}
The proof of Lipschitz continuity in $x$ is straightforward, while the H\"older continuity in $t$ can be shown as in Proposition 4.1. and Proposition 4.6. in \cite{CG}.\\
It remains to prove the uniform Lipschitz continuity in $p$. Since we have convexity in $p$, it is sufficient to establish the Lipschitz continuity with respect to $p$ on the extreme points $e_i$. Observe that $\mathcal{P}(t,e_i)$ consists in the single probability measure $\delta(e_i)\otimes\mathbb{P}_0$, where $\delta(e_i)$ is the measure under which $\bold p$ is constant and equal to $e_i$ and $\mathbb{P}_0$ is a Wiener measure.\\
Assume $W^+(t,x,e_i)-W^+(t,x,p)>0$. For $\epsilon>0$ let $\mathbb{P}^\epsilon\in \mathcal{P}(t,p)$ be $\epsilon$-optimal for $W^+(t,x,p)$. Then
\begin{equation}
\begin{array}{rcl}
&&W^+(t,x,e_i)-W^+(t,x,p)-3\epsilon\\
 \ \\
&&\leq   \textnormal{essinf}_{\tau\in \bar{\mathcal{T}}(t)}\textnormal{esssup}_{\sigma \in \bar{\mathcal{T}}(t)} J(t,x,\tau,\sigma,\delta(e_i)\otimes\mathbb{P}_0)_{t-}\\
\ \\
&&\ \ \ \ \ \ \ -\textnormal{essinf}_{\tau\in \bar{\mathcal{T}}(t)}\textnormal{esssup}_{\sigma \in \bar{\mathcal{T}}(t)} J(t,x,\tau,\sigma,\mathbb{P}^\epsilon)_{t-} -2\epsilon.
\end{array}
\end{equation}
Choose now $\bar{\tau}\in\bar{\mathcal{T}}(t)$ to be $\epsilon$-optimal for  $ \textnormal{essinf}_{\tau\in \bar{\mathcal{T}}(t)} \textnormal{esssup}_{\sigma \in \bar{\mathcal{T}}(t)} J(t,x,\tau,\sigma,\mathbb{P}^\epsilon)_{t-}$ and $\bar{\sigma}\in\bar{\mathcal{T}}(t)$ to be $\epsilon$-optimal for  $\textnormal{esssup}_{\sigma \in \bar{\mathcal{T}}(t)}  J(t,x,\bar \tau,\sigma,,\delta(e_i)\otimes\mathbb{P}_0))_{t-}$. Then we have with (7.1)
\begin{equation}
\begin{array}{rcl}
&& W^+(t,x,e_i)-W^+(t,x,p)-3\epsilon\\
&&\leq \mathbb{E}_{{\mathbb{P}^{\epsilon}}}\bigg[\langle e_i-\bold p_{\bar{\sigma}},f({\bar{\sigma}},X^{t,x}_{\bar{\sigma}})\rangle1_{{\bar{\sigma}}<{\bar{\tau}}\leq T}+\langle e_i-\bold p_{\bar{\tau}}, h(\bar \tau,X^{t,x}_{\bar{\tau}})\rangle1_{\bar{ \tau}\leq\bar {\sigma},\bar{\tau}<T}\\
&&\ \ \ \ \ \ \ \ \ \ \ \ \ \ \ \ \ \ \ \ \ \ \ \ \ \ \ \ \ + \langle e_i-\bold p_T,g(X^{t,x}_T)\rangle 1_{\bar \sigma=\bar \tau=T}|\mathcal{F}_{t-}\bigg].
\end{array}
\end{equation}
Since for all $p\in\Delta(I)$  $0\leq|p-e_i|\leq c(1- p_i)$ we have by the boundedness of the coefficients with (7.2) and the fact that $\bold p$ is a ${{\mathbb{P}^{\epsilon}}}$-martingale with mean $p$
\begin{equation*}
\begin{array}{rcl}
&& W^+(t,x,e_i)-W^+(t,x,p)-3\epsilon\\
\ \\
&\leq& c \left(1- \mathbb{E}_{{\mathbb{P}^{\epsilon}}}\left[(\bold p_{\bar{\sigma}})_i1_{{\bar{\sigma}}<{\bar{\tau}}\leq T}+(\bold p_{\bar{\tau}})_i1_{\bar{ \tau}\leq\bar {\sigma},\bar{\tau}<T}+(\bold p_{T})_i 1_{\bar \sigma=\bar \tau=T}|\mathcal{F}_{t-}\right]\right)\\
\ \\
&\leq&c(1- p_i).
\end{array}
\end{equation*}
Using now
\begin{equation*}
1- p_i\leq c  \sum_{j} |(p)_j-\delta_{ij}|\leq c \sqrt{I}|p-e_i|,
\end{equation*}
the claim follows since $\epsilon$ can be chosen arbirarily small. The case $W^+(t,x,p)-W^+(t,x,e_i)>0$ is immediate.\\
The Lipschitz continuity of $W^-$ in $p$ can be established by similar arguments.
\end{proof}

\subsection{Subdynamic programming for $W^+$}

\begin{thm}
Let $(\bar t,\bar x,\bar p)\in[0,T]\times \mathbb{R}^d\times\Delta(I)$. Then for all $t\in[\bar t,T]$
\begin{equation}
\begin{array}{l}
W^+ (\bar t,\bar x,\bar p)\\
\ \ \ \leq\textnormal{essinf}_{\mathbb{P}\in\mathcal{P}(\bar t,\bar p)}  \textnormal{essinf}_{\tau\in \bar{\mathcal{T}}(\bar t,t)}\textnormal{esssup}_{\sigma \in \bar{\mathcal{T}}(\bar t,t)}\mathbb{E}_{\mathbb{P}}\bigg[\langle \bold p_\sigma,f(\sigma, X^{\bar t,\bar x}_{\sigma}) 1_{\sigma<\tau,\sigma<t}\rangle\\
 \ \ \ \ \ \ \ \ \ \ \ \ \ \ \ \ \ \ \ \ \ +\langle \bold p_\tau, h(\tau,X^{\bar t,\bar x}_{\tau} )1_{\tau\leq\sigma,\tau<t}\rangle+ W^+(t,X^{\bar t,\bar x}_t,\bold p_{t-}) 1_{\tau=\sigma=t}|\mathcal{F}_{\bar t-}\bigg].
\end{array}
\end{equation}
\end{thm}

\begin{proof}
Let $\mathbb{P}\in\mathcal{P}^f(t,p)$, $t\in[\bar t,T]$. By assumption there exist $S=\{p^1,\ldots,p^k\}$, such that $\mathbb{P}[\bold p_{t-}\in S]=1$.
Furthermore let $(A_l)_{l\in\mathbb{N}}$ be a partition of $\mathbb{R}^d$ by Borel sets, such that diam$(A_l)\leq \bar \epsilon$ and choose for any $l\in\mathbb{N}$ some $y^l\in A_l$.\\

Define for any $l,m$ measures $\mathbb{P}^{l,m}\in\mathcal{P}^f(t,p^m)$, such that they are $\epsilon$-optimal for $W^+(t,p^m,y^l)$ and $\epsilon$-optimal stopping times $\tau^{l,m}$.
We define the probablility measure ${\mathbb{P}}^\epsilon$, such that on $\Omega=\Omega_{0,t}\times\Omega_{t}$
\begin{equation}
\begin{array}{l}
 \mathbb{P}^\epsilon=(\mathbb{P}|_{\Omega_{0,t}})\otimes \hat{\mathbb{P}},
\end{array}
\end{equation}
where for all $A\in\mathcal{B}(\Omega_t)$:
\begin{equation*}
\hat{\mathbb{P}}[A]=\sum_{m=1}^k\sum_{l=1}^\infty \mathbb{P}[X_{t}^{\bar t,\bar x}\in A^l,\bold p_{t-}=p_m]\mathbb{P}^{l,m}[A],
\end{equation*}
and the stopping time
\begin{equation}
\hat \tau=\begin{cases}
 \tau & \text{on } \{ \tau<t\}\\
\tau^{l,m} & \text{on } \{\bar \tau \geq t, X^{\bar t,\bar x}_t\in A^l,\bold p_{t-}=p^m \}. \\
\end{cases}
\end{equation}
Note that by definition $(B_s)_{s\in[\bar t,T]}$ is a Brownian motion under $\mathbb{P}^\epsilon$. Also $(\bold p_s)_{s\in[\bar t,T]}$ is a martingale, since for $ t\leq r\leq s\leq T$
\[
\mathbb{E}_{\mathbb{P}^\epsilon}[\bold p_s|\mathcal{F}_{r}]=\sum_{m=1}^k \sum_{l=1}^\infty 1_{\{X_{t}^{t,x}\in A^l,\bold p_{t-}=p^m\}}\mathbb{E}_{\mathbb{P}^{l,m}}[\bold p_s|\mathcal{F}_{r}]=\sum_{m=1}^k \sum_{l=1}^\infty 1_{\{X_{t}^{\bar t,\bar x}\in A^l,\bold p_{t-}=p^m\}}\bold p_r=\bold p_r.
\]
Furthermore the remaining conditions of Definition 6.1. are obviously met, hence $\mathbb{P}^\epsilon\in\mathcal{P}^{f}(t,p)$. By the definition of $W^+$ we have 
\begin{equation}
\begin{array}{rcl}
&&W^+ (\bar t,\bar x,\bar p)\\
\ \\
&&\ \ \ \leq  \textnormal{esssup}_{\sigma \in \bar{\mathcal{T}}(\bar t)}\mathbb{E}_{\mathbb{P}^\epsilon}\bigg[\langle \bold p_\sigma,f(\sigma,X^{\bar t,\bar x}_\sigma)\rangle1_{\sigma<\hat \tau,\sigma<T}+\langle \bold p_{\hat \tau}, h(\hat \tau,X^{\bar t,\bar x}_{\hat \tau})\rangle1_{ \hat\tau\leq\sigma,\hat\tau<T}\\
\ \\
&&\ \ \ \ \ \ \ \ \ \ \ \ \ \ \ \ \ \ \ \ \ \ \ \ \ \ \ \ \ \ \ \ \ \ \ \ + \langle \bold p_T,g(X^{t,x}_T)\rangle 1_{\sigma=\hat \tau=T}|\mathcal{F}_{t-}\bigg].
\end{array}
\end{equation}
Note that using the Lipschitz continuity of $W^+$ we have for any $\sigma\in\bar{\mathcal{T}}(\bar t)$
\begin{equation*}
\begin{array}{l}
\mathbb{E}_{\mathbb{P}^\epsilon}\bigg[\langle \bold p_\sigma,f(\sigma,X^{\bar t,\bar x}_\sigma)\rangle1_{t \leq\sigma<\hat \tau,\sigma<T}+\langle \bold p_{\hat \tau}, h(\hat \tau,X^{\bar t,\bar x}_{\hat \tau})\rangle1_{t\leq\hat\tau\leq\sigma,\hat \tau<T}+\langle \bold p_T,g(X^{\bar t,\bar x}_T)\rangle 1_{\sigma=\tau=T}|\mathcal{F}_{t-}\bigg]\\
\ \\
\ \ \ \leq \mathbb{E}_{\mathbb{P}^\epsilon}\bigg[W^+(t,y^l,p^m)  1_{\{X_{t}^{\bar t,\bar x}\in A^l,\bold p_{t-}=p^m\}}1_{\{\sigma\geq t,\hat \tau \geq t\}} |\mathcal{F}_{t-}\bigg]+c \delta+ 2\epsilon\\
\ \\
\ \ \ \leq \mathbb{E}_{\mathbb{P}^\epsilon}\bigg[W^+(t,X_{t}^{\bar t,\bar x},\bold p_{t-})1_{\{\sigma\geq t,\hat \tau \geq t\}} |\mathcal{F}_{t-}\bigg]+ 2 c \delta+2\epsilon.
\end{array}
\end{equation*}
Hence we have with (7.6)
\begin{equation*}
\begin{array}{rcl}
&&W^+ (\bar t,\bar x,\bar p)\\
\ \\
&&\ \ \ \leq  \textnormal{esssup}_{\sigma \in \bar{\mathcal{T}}(\bar t,t)}\mathbb{E}_{\mathbb{P}^\epsilon}\bigg[\langle \bold p_\sigma,f(\sigma,X^{\bar t,\bar x}_\sigma)\rangle1_{\sigma<\tau,\sigma<t}+\langle \bold p_{\tau}, h(\tau,X^{\bar t,\bar x}_{\tau})\rangle1_{\tau\leq\sigma,\tau<t}\\
\ \\
&&\ \ \ \ \ \ \ \ \ \ \ \ \ \ \ \ \ \ \ \ \ \ \ \ \ \ \ \ \ \ \ \ \ \ \ \ +W^+(t,X_{t}^{\bar t,\bar x},\bold p_{t-})1_{\{\sigma= \tau=t\}}|\mathcal{F}_{t-}\bigg]+ 2 c \delta+ 2\epsilon.
\end{array}
\end{equation*}
Now choosing $\mathbb{P}$, $\tau\in\bar{\mathcal{T}}(\bar t,t)$ such that they are $\epsilon$ optimal for the right hand side of (7.3) gives the desired result.
\end{proof}

\subsection{Superdynamic programming for $W^-$}

\begin{thm}
Let $(\bar t,\bar x,\bar p)\in[0,T]\times \mathbb{R}^d\times\Delta(I)$. Then for all $t\in[\bar t,T]$
\begin{equation}
\begin{array}{l}
W^- (\bar t,\bar x,\bar p)\\
\ \ \ \geq\textnormal{essinf}_{\mathbb{P}\in\mathcal{P}(\bar t,\bar p)}  \textnormal{esssup}_{\sigma \in \bar{\mathcal{T}}(\bar t,t)}\textnormal{essinf}_{\tau\in \bar{\mathcal{T}}(\bar t,t)}\mathbb{E}_{\mathbb{P}}\bigg[\langle \bold p_\sigma,f(\sigma, X^{\bar t,\bar x}_{\sigma})\rangle 1_{\sigma<\tau,\sigma<t}\\
 \ \ \ \ \ \ \ \ \ \ \ \ \ \ \ \ \ \ \ \ \ +\langle \bold p_\tau, h(\tau,X^{\bar t,\bar x}_{\tau} )1_{\tau\leq\sigma,\tau<t}\rangle+ W^-(t,X^{\bar t,\bar x}_t,\bold p_{t-}) 1_{\tau=\sigma=t}|\mathcal{F}_{\bar t-}\bigg].
\end{array}
\end{equation}
\end{thm}

\begin{proof}
We choose a $\mathbb{P}^\epsilon\in\mathcal{P}^f(\bar t,\bar p)$ to be $\epsilon$-optimal for $W^-(\bar t,\bar x,\bar p)$,
\begin{equation}
\begin{array}{l}
W^- (\bar t,\bar x,\bar p)\\
\ \ \ \geq \textnormal{essup}_{\sigma\in\bar{\mathcal{T}}(\bar t)}\textnormal{essinf}_{\tau\in \bar{\mathcal{T}}(\bar t)}\mathbb{E}_{\mathbb{P}^\epsilon}\bigg[\langle \bold p_\sigma,f(\sigma,X^{t,x}_\sigma)\rangle1_{\sigma<\tau,\sigma<T}+\langle \bold p_\tau, h(\tau,X^{t,x}_\tau)\rangle1_{ \tau\leq\sigma,\tau<T}\\
\ \ \ \ \ \ \ \ \ \ \ \ \ \ \ \ \ \ \ \ \ \ \ \ \ \ \ \ \ + \langle \bold p_T,g(X^{t,x}_T)\rangle 1_{\sigma=\tau=T}|\mathcal{F}_{\bar t-}\bigg]-\epsilon.
\end{array}
\end{equation}
By assumption there exist $S=\{p^1,\ldots,p^k\}$, such that $\mathbb{P}^\epsilon[\bold p_{t-}\in S]=1$.
Furthermore let $(A_l)_{l\in\mathbb{N}}$ be a partition of $\mathbb{R}^d$ by Borel sets, such that diam$(A_l)\leq \bar \epsilon$ and choose for any $l\in\mathbb{N}$ some $y^l\in A_l$.\\
With the help of $\mathbb{P}^\epsilon$ define $\mathbb{P}^{l,m}$ as
\begin{equation}
\mathbb{P}^{l,m}=(\mathbb{P}_0\otimes\delta(p^m))\otimes\hat{\mathbb{P}}^{l,m},
\end{equation}
where $\delta(p^m)$ denotes the measure under which $\bold p$ is constant and equal to $p^m$, $\mathbb{P}_0$ is a Wiener measure on $\Omega_{0,t}$ and for all $A\in\mathcal{B}(\Omega_t)$ 
\begin{equation*}
\hat{\mathbb{P}}^{l,m}=\mathbb{P}^\epsilon[\bold p_{t-}=p^m,X^{\bar t,\bar x}_t\in A^l]\mathbb{P}^{\epsilon}[A|\bold p_{t-}=p^m,X^{\bar t,\bar x}_t\in A^l].
\end{equation*}
Furthermore define stopping times $\sigma^{l,m}\in\bar{\mathcal{T}}(t)$ which are $\epsilon$-optimal for
\begin{equation}
\begin{array}{l}
 \textnormal{esssup}_{\sigma\in\bar{\mathcal{T}}(t)} \textnormal{essinf}_{\tau\in\bar{\mathcal{T}}(t)}\mathbb{E}_{\mathbb{P}^{l,m}}\bigg[\langle \bold p_\sigma,f(\sigma,X^{t,y^l}_\sigma)\rangle1_{\sigma<\tau,\sigma<T}\\
\ \ \ \ \ \ \ \ \ \ \ \ \ \ \ \ +\langle \bold p_{\tau}, h(\tau,X^{t,y^l}_{ \tau})\rangle1_{\tau\leq\sigma,\tau<T}+\langle \bold p_T,g(X^{\bar t,\bar x}_T)\rangle 1_{\sigma=\tau=T}|\mathcal{F}_{t-}\bigg],
\end{array}
\end{equation}
which implies that for all $\tau\in\bar{\mathcal{T}}(t)$
\begin{equation}
\begin{array}{l}
\mathbb{E}_{\mathbb{P}^{l,m}}\bigg[\langle \bold p_{\sigma^{l,m}},f(\sigma^{l,m},X^{t,y^l}_{\sigma^{l,m}})\rangle1_{{\sigma}^{l,m}<\tau,{\sigma}^{l,m}<T}+\langle \bold p_{\tau}, h(\tau,X^{\bar t,\bar x}_{ \tau})\rangle1_{\tau\leq\sigma^{l,m},\tau<T}\\
\ \ \ \ \ \ \ \ \ \ \ \ \ \ \ \ +\langle \bold p_T,g(X^{\bar t,\bar x}_T)\rangle 1_{\sigma^{l,m}=\tau=T}|\mathcal{F}_{t-}\bigg]\\
\geq \textnormal{esssup}_{\sigma\in\bar{\mathcal{T}}(t)} \textnormal{essinf}_{\tau\in\bar{\mathcal{T}}(t)}\mathbb{E}_{\mathbb{P}^{l,m}}\bigg[\langle \bold p_\sigma,f(\sigma,X^{t,y^l}_\sigma)\rangle1_{\sigma<\tau,\sigma<T}\\
\ \ \ \ \ \ \ \ \ \ \ \ \ \ \ \ +\langle \bold p_{\tau}, h(\tau,X^{t,y^l}_{ \tau})\rangle1_{\tau\leq\sigma,\tau<T}+\langle \bold p_T,g(X^{\bar t,\bar x}_T)\rangle 1_{\sigma=\tau=T}|\mathcal{F}_{t-}\bigg]-\epsilon
\end{array}
\end{equation}
and
\begin{equation}
\begin{array}{l}
\mathbb{E}_{\mathbb{P}^{l,m}}\bigg[\langle \bold p_{\sigma^{l,m}},f(\sigma^{l,m},X^{t,y^l}_{\sigma^{l,m}})\rangle1_{{\sigma}^{l,m}<\tau,{\sigma}^{l,m}<T}+\langle \bold p_{\tau}, h(\tau,X^{\bar t,\bar x}_{ \tau})\rangle1_{\tau\leq {\sigma}^{l,m},\tau<T}\\
\ \ \ \ \ \ \ \ \ \ \ \ \ \ \ \ +\langle \bold p_T,g(X^{\bar t,\bar x}_T)\rangle 1_{\sigma^{l,m}=\tau=T}|\mathcal{F}_{t-}\bigg]\\
\geq
\textnormal{essinf}_{\mathbb{P}\in\mathcal P({t,p^m})}\textnormal{esssup}_{\sigma\in\bar{\mathcal{T}}(t)} \textnormal{essinf}_{\tau\in\bar{\mathcal{T}}(t)}\mathbb{E}_{P}\bigg[\langle \bold p_\sigma,f(\sigma,X^{t,y^l}_\sigma)\rangle1_{\sigma<\tau,\sigma<T}\\
\ \ \ \ \ \ \ \ \ \ \ \ \ \ \ \ +\langle \bold p_{\tau}, h(\tau,X^{t,y^l}_{ \tau})\rangle1_{\tau\leq\sigma,\tau<T}+\langle \bold p_T,g(X^{\bar t,\bar x}_T)\rangle 1_{\sigma=\tau=T}|\mathcal{F}_{t-}\bigg]-\epsilon\\
= W^-(t,p^m,y^l)-\epsilon.
\end{array}
\end{equation}

For any $\sigma\in\bar{\mathcal{T}}(\bar t)$ define
\begin{equation}
\hat \sigma=\begin{cases}
 \sigma & \text{on } \{ \sigma<t\}\\
\sigma^{l,m} & \text{on } \{\sigma \geq t, X^{\bar t,\bar x}_t\in A^l,\bold p_{t-}=p^m \}. \\
\end{cases}
\end{equation}

Note that using the Lipschitz continuity of the coefficients and $W^-$ and the definition of $\hat \sigma$ and $\mathbb{P}^{l,m}$ we have for any $\tau\in\bar{\mathcal{T}}(\bar t)$ 
\begin{equation*}
\begin{array}{l}
\mathbb{E}_{\mathbb{P}^\epsilon}\bigg[\langle \bold p_{\hat\sigma},f({\hat\sigma},X^{t,x}_{\hat\sigma})\rangle1_{t\leq{\hat\sigma}<\tau,\hat\sigma<T}+\langle \bold p_\tau, h(\tau,X^{t,x}_\tau)\rangle1_{t\leq\tau\leq{\hat\sigma},\tau<T}\\
\ \ \ \ \ \ \ \ \ \ \ \ \ \ \ \ \ \ \ \ \ \ \ \ \ \ \ \ \ + \langle \bold p_T,g(X^{t,x}_T)\rangle 1_{\hat \sigma=\tau=T,\hat\sigma,\tau\geq t}|\mathcal{F}_{\bar t-}\bigg]\\
=\mathbb{E}_{\mathbb{P}^\epsilon}\bigg[\mathbb{E}_{\mathbb{P}^\epsilon}\bigg[\langle \bold p_{\hat \sigma},f(\hat \sigma,X^{t,x}_{\hat \sigma})\rangle1_{t\leq\hat \sigma<\tau,\hat \sigma<T}+\langle \bold p_\tau, h(\tau,X^{t,x}_\tau)\rangle1_{t\leq\tau\leq{\hat\sigma},\tau<T}\\
\ \ \ \ \ \ \ \ \ \ \ \ \ \ \ \ \ \ \ \ \ \ \ \ \ \ \ \ \ + \langle \bold p_T,g(X^{t,x}_T)\rangle 1_{\hat \sigma=\tau=T,\hat \sigma,\tau\geq t}|\mathcal{F}_{t-}\bigg]|\mathcal{F}_{\bar t-}\bigg]\\
\geq \mathbb{E}_{\mathbb{P}^\epsilon}\bigg[1_{\sigma,\tau \geq t, X^{\bar t,\bar x}_t\in A^l,\bold p_{t-}=p^m }\mathbb{E}_{\mathbb{P}^{l,m}}\bigg[\langle \bold p_{\sigma^{l,m}},f(\sigma^{l,m},X^{t,y^l}_{\sigma^{l,m}})\rangle1_{{\sigma}^{l,m}<\tau,\sigma^{l,m}<T}\\
\ \ \ \ \ \ \ \ \ \ \ \ \ \ \ \ +\langle \bold p_{\tau}, h(\tau,X^{\bar t,\bar x}_{ \tau})\rangle1_{\tau\leq {\sigma}^{l,m},\tau<T}+\langle \bold p_T,g(X^{\bar t,\bar x}_T)\rangle 1_{\sigma^{l,m}=\tau=T}|\mathcal{F}_{t-}\bigg]|\mathcal{F}_{\bar t-}\bigg]\\
\ \ \ \ \ \ \ \ -c\delta\\
\geq \mathbb{E}_{\mathbb{P}^\epsilon}\left[1_{\sigma,\tau \geq t, X^{\bar t,\bar x}_t\in A^l,\bold p_{t-}=p^m } W(t,y^l,p^m)|\mathcal{F}_{\bar t-}\right]-c\delta-\epsilon\\
\ \\
\geq \mathbb{E}_{\mathbb{P}^\epsilon}\left[ 1_{\sigma,\tau \geq t}W(t,X^{\bar t,\bar x}_t,\bold p_{t-})|\mathcal{F}_{\bar t-}\right]-2c\delta-\epsilon.
\end{array}
\end{equation*}
This gives with (7.12) for any $\sigma\in\bar{\mathcal{T}}(\bar t,t)$
\begin{equation}
\begin{array}{l}
W^- (\bar t,\bar x,\bar p)\\
\ \ \ \geq \textnormal{essinf}_{\tau\in \bar{\mathcal{T}}(\bar t,t)}\mathbb{E}_{\mathbb{P}^\epsilon}\bigg[\langle \bold p_\sigma,f(\sigma,X^{t,x}_\sigma)\rangle1_{\sigma<\tau,\sigma<t}+\langle \bold p_\tau, h(\tau,X^{t,x}_\tau)\rangle1_{\tau\leq\sigma,\tau<t}\\
\ \ \ \ \ \ \ \ \ \ \ \ \ \ \ \ \ \ \ \ \ \ \ \ \ \ \ \ \ +W(t,X^{\bar t,\bar x}_t,\bold p_{t-}) 1_{\sigma=\tau=t}|\mathcal{F}_{\bar t-}\bigg]-2c\delta-\epsilon.
\end{array}
\end{equation}
So in particular when choosing $\bar \sigma$ $\epsilon$-optimal for
\begin{equation}
\begin{array}{l}
\textnormal{esssup}_{\sigma\in \bar{\mathcal{T}}(\bar t,t)}\textnormal{essinf}_{\tau\in \bar{\mathcal{T}}(\bar t,t)}\mathbb{E}_{\mathbb{P}^\epsilon}\bigg[\langle \bold p_\sigma,f(\sigma,X^{t,x}_\sigma)\rangle1_{\sigma<\tau,\sigma<t}\\
\ \ \ \ \ \ \ \ \ \ \ \ \ \ \ \ \ \ \ \ \ \ \ \ +\langle \bold p_\tau, h(\tau,X^{t,x}_\tau)\rangle1_{ \tau\leq\sigma,\tau<t}+W(t,X^{\bar t,\bar x}_t,\bold p_{t-}) 1_{\sigma=\tau=t}|\mathcal{F}_{\bar t-}\bigg]
\end{array}
\end{equation}
we get
\begin{equation*}
\begin{array}{l}
W^- (\bar t,\bar x,\bar p)\\
\ \ \ \geq \textnormal{esssup}_{\sigma\in \bar{\mathcal{T}}(\bar t,t)}\textnormal{essinf}_{\tau\in \bar{\mathcal{T}}(\bar t,t)}\mathbb{E}_{\mathbb{P}^\epsilon}\bigg[\langle \bold p_\sigma,f(\sigma,X^{t,x}_\sigma)\rangle1_{\sigma<\tau,\sigma<t}+\langle \bold p_\tau, h(\tau,X^{t,x}_\tau)\rangle1_{\tau\leq\sigma,\tau<t}\\
\ \ \ \ \ \ \ \ \ \ \ \ \ \ \ \ +\langle \bold p_\tau, h(\tau,X^{t,x}_\tau)\rangle1_{t>\sigma\geq \tau}+W(t,X^{\bar t,\bar x}_t,\bold p_{t-}) 1_{\sigma=\tau=t}|\mathcal{F}_{\bar t-}\bigg]-2c\delta-2\epsilon
\end{array}
\end{equation*}
and the claim follows by taking the essential infimum in $\mathbb{P}\in\mathcal{P}(\bar t,\bar p)$ since $\delta$ and $\epsilon$ can be chosen arbitrarily small.

\end{proof}

\section{Viscosity solution property $W^+,W^-$}

\subsection{Subsolution property of $W^+$}

\begin{thm}
$W^+$ is a viscosity subsolution to the obstacle problem
\begin{equation}
\begin{array}{l}
	\max\bigg\{\max\{\min\{(-\frac{\partial } {\partial t}-\mathcal{L})[w],w-\langle f(t,x),p\rangle\},\\
	\ \ \ \ \ \ \ \ \ \ \ \ \ \ \ \ \ \ \ \ \ \ \ \  \ \ \ \ \ w-\langle h(t,x),p\rangle\}, -\lambda_{\min}\left(p,\frac{\partial ^2 w}{\partial p^2}\right)\bigg\}=0
\end{array}
\end{equation}
with terminal condition $w(T,x,p)=\sum_{i=1,\ldots,I}p_ig_i(x)$.
\end{thm}

\begin{proof}
Let $\phi:[0,T]\times\mathbb{R}^d\times\Delta(I)\rightarrow \mathbb{R}$ be a test function such that $W^+-\phi$ has a strict global maximum at $(\bar{t},\bar x,\bar p)\in[0,T)\times\mathbb{R}^d\times\textnormal{Int}(\Delta (I))$ with $W(\bar{t},\bar x,\bar p)-\phi(\bar{t},\bar x,\bar p)=0$.
We have to show, that
\begin{equation}
\begin{array}{l}
	\max\big\{\max\{\min\{(-\frac{\partial } {\partial t}-\mathcal{L})[\phi],\phi-\langle f(t,x),p\rangle\},\\
	\ \\
	\ \ \ \ \ \ \ \ \ \ \ \ \ \ \ \ \ \ \ \ \ \ \ \  \ \ \ \ \ \phi-\langle h(t,x),p\rangle\}, -\lambda_{\min}\left(p,\frac{\partial ^2 w}{\partial p^2}\right)\big\}\leq0
\end{array}
\end{equation}
at $(\bar{t},\bar x,\bar p)$.\\
By Proposition 7.2 $W^+$ is convex in $p$. So since $\bar p\in\textnormal{Int}(\Delta (I))$, we have that 
\[-\lambda_{\min}\left(\frac{\partial ^2 \phi}{\partial p^2}(\bar{t},\bar x,\bar p)\right)\leq 0.\] 
So it remains to show, that
\begin{equation}
\begin{array}{l}
	\max\{\min\{(-\frac{\partial } {\partial t}-\mathcal{L})[\phi],\phi-\langle f(t,x),p\rangle\},\phi-\langle h(t,x),p\rangle\}\leq0
\end{array}
\end{equation}
at $(\bar{t},\bar x,\bar p)$. Note that the subdynamic programming for $W^+$ implies for $\mathbb{P}=\mathbb{P}_0\otimes\delta(\bar p)$ in particular
\begin{equation*}
\begin{array}{l}
W^+ (\bar t,\bar x,\bar p)\\
\ \ \ \leq  \textnormal{essinf}_{\tau\in \bar{\mathcal{T}}(\bar t,t)}\textnormal{esssup}_{\sigma \in \bar{\mathcal{T}}(\bar t,t)}\mathbb{E}_{\mathbb{P}}\bigg[\langle \bar p,f(\sigma, X^{\bar t,\bar x}_{\sigma}) 1_{\sigma<\tau,\sigma<t}\rangle\\
 \ \ \ \ \ \ \ \ \ \ \ \ \ \ \ \ \ \ \ \ \ +\langle \bar p, h(\tau,X^{\bar t,\bar x}_{\tau} )1_{\tau\leq\sigma,\tau<t}\rangle+ W^+(t,X^{\bar t,\bar x}_t,\bar p) 1_{\tau=\sigma=t}|\mathcal{F}_{\bar t-}\bigg].
\end{array}
\end{equation*}
So (8.2) follows by the standard arguments we mentioned already in the proof of Theorem 5.1.

\end{proof}

\subsection{Supersolution property of $W^-$}

\begin{thm}
$W^-$ is a viscosity supersolution to the obstacle problem
\begin{equation}
\begin{array}{l}
	\max\bigg\{\max\{\min\{(-\frac{\partial } {\partial t}-\mathcal{L})[w],w-\langle f(t,x),p\rangle\},\\
	\ \ \ \ \ \ \ \ \ \ \ \ \ \ \ \ \ \ \ \ \ \ \ \  \ \ \ \ \ w-\langle h(t,x),p\rangle\}, -\lambda_{\min}\left(p,\frac{\partial ^2 w}{\partial p^2}\right)\bigg\}=0
\end{array}
\end{equation}
with terminal condition $w(T,x,p)=\sum_{i=1,\ldots,I}p_ig_i(x)$.
\end{thm}

\begin{proof}
Let $\phi:[0,T]\times\mathbb{R}^d\times\Delta(I)\rightarrow \mathbb{R}$ be a smooth test function with uniformly bounded derivatives such that $W^--\phi$ has a strict global minimum at $(\bar{t},\bar x,\bar p)\in[0,T)\times\mathbb{R}^d\times\Delta(I)$ with $W^-(\bar{t},\bar x,\bar p)-\phi(\bar{t},\bar x,\bar p)=0$. We have to show
 \begin{equation}
				\begin{array}{l}
	\max\big\{\max\{\min\{(-\frac{\partial } {\partial t}-\mathcal{L})[\phi],\phi-\langle f(t,x),p\rangle\},\\
	\ \\
	 \ \ \ \ \ \ \ \ \ \ \ \ \ \ \  \ \ \ \ \ \ \ \ \ \ \ \ \ \ \ \phi-\langle h(t,x),p\rangle\}, -\lambda_{\min}\left(p,\frac{\partial ^2 \phi}{\partial p^2}\right)\big\}\geq0
		\end{array}
\end{equation}
at $(\bar{t},\bar x,\bar p)$. If
\begin{equation*}
 \lambda_{\min}\left(p,\frac{\partial ^2 \phi}{\partial p^2}\right)\leq0
\end{equation*}
at $(\bar{t},\bar x,\bar p)$ (8.5) obviously holds. So we assume in the subsequent steps strict convexity of $\phi$ in $p$ at $(\bar{t},\bar x,\bar p)$, i.e. there exist $\delta,\eta>0$ such that for all $z\in T_{\Delta(I)(\bar p)}$
\begin{eqnarray}
\langle \frac{\partial^2 \phi}{\partial p^2}(t,x,p)z,z\rangle>4\delta|z|^2\ \ \ \ \ \ \forall (t,x,p)\in B_\eta(\bar t,\bar x,\bar p).
\end{eqnarray}
Since $\phi$ is a test function for a purely local viscosity notion, one can modify it outside a neighborhood of $(\bar t,\bar x,\bar p)$ such that for all $(s, x)\in [\bar t,T]\times\mathbb{R}^d$ the function $\phi(s,x,\cdot)$ is convex on the whole convex domain $\Delta(I)$. Thus for any $p\in\Delta(I)$ we have that
\begin{eqnarray}
W^-(t,x,p)\geq \phi(t,x,p)\geq \phi(t,x,\bar p)+\langle\frac{\partial \phi }{\partial p}(t,x,\bar p),p-\bar p\rangle.
\end{eqnarray}
\ \\
\emph{Step 1: Estimate for $\bold p$.}\\
As in (4.14) we have with (8.6) a stronger estimate, namely there exist $\delta, \eta>0$ such that for all $p\in\Delta(I)$, $t\in[\bar t,\bar t+\eta]$, $x\in B_\eta(\bar x)$
\begin{equation}
W^-(t,x,p)\geq\phi(t,x,\bar p)+\langle \frac{\partial\phi}{\partial p} (t,x,\bar p),p-\bar p\rangle+ \delta |p-\bar p|^2.
\end{equation}

As in the proof of Theorem 4.1. we can set in the dynamic programming for $W^-$ $\sigma=t$ to get
\begin{equation}
\begin{array}{l}
W^- (\bar t,\bar x,\bar p)\\
\geq\textnormal{essinf}_{\mathbb{P}\in\mathcal{P}(\bar t,\bar p)}  \textnormal{essinf}_{\tau\in \bar{\mathcal{T}}(\bar t,t)}\mathbb{E}_{\mathbb{P}}\bigg[\langle \bold p_\tau, h(\tau,X^{\bar t,\bar x}_{\tau} )1_{\tau<t}\rangle+ W^-(t,X^{\bar t,\bar x}_t,\bold p_{t-}) 1_{\tau=t}|\mathcal{F}_{\bar t-}\bigg].
\end{array}
\end{equation}
So for $\epsilon(t-\bar t)$-optimal $\mathbb{P}^\epsilon\in\mathcal{P}^f(t,p)$ and a $\epsilon(t-\bar t)$-optimal stopping time $\tau^\epsilon$ we have
\begin{equation}
\begin{array}{l}
W^- (\bar t,\bar x,\bar p)\\
\ \\
\ \ \ \geq\mathbb{E}_{\mathbb{P}^\epsilon}\bigg[\langle \bold p_{\tau^\epsilon}, h(\tau^\epsilon,X^{\bar t,\bar x}_{\tau^\epsilon} )\rangle 1_{\tau^\epsilon<t}+ W^-(t,X^{\bar t,\bar x}_t,\bold p_{t-}) 1_{\tau^\epsilon=t}|\mathcal{F}_{\bar t-}\bigg]-2\epsilon(t-\bar t)\\
\ \\
\ \ \ = \mathbb{E}_{\mathbb{P}^\epsilon}\bigg[\langle \bold p_{\tau^\epsilon-}, h(\tau^\epsilon,X^{\bar t,\bar x}_{\tau^\epsilon} )\rangle 1_{\tau^\epsilon<t}+ W^-(t,X^{\bar t,\bar x}_t,\bold p_{t-}) 1_{\tau^\epsilon=t}|\mathcal{F}_{\bar t-}\bigg]-2\epsilon(t-\bar t)\\
\ \\
\ \ \ \geq\mathbb{E}_{\mathbb{P}^\epsilon}\bigg[W(\tau^\epsilon,X^{\bar t,\bar x}_{\tau^\epsilon},\bold p_{\tau^\epsilon-})|\mathcal{F}_{\bar t-}\bigg]-2\epsilon(t-\bar t),
\end{array}
\end{equation}
since $\langle p, h(t,x)\rangle\geq W^-(t,x,p)$ for all $(t,x,p)\in[0,T]\times\mathbb{R}^d\times\Delta(I)$. Using (8.7) and (8.8) we get since $W^- (\bar t,\bar x,\bar p)=\phi(\bar t,\bar x,\bar p)$
\begin{equation}
\begin{array}{l}
0\geq\mathbb{E}_{\mathbb{P}^\epsilon}\bigg[\phi(\tau^\epsilon,X^{\bar t,\bar x}_{\tau^\epsilon},\bar p)-\phi(\bar t,\bar x,\bar p)-\langle\frac{\partial \phi}{\partial p}(\tau^\epsilon,X^{\bar t,\bar x}_{\tau^\epsilon},\bar p),\bold p_{\tau^\epsilon-}-\bar p\rangle\\
\ \\
\ \ \ \ \ \ \ \ \ \ \ \ \ \ \ \ \ +\delta 1_{\{|X^{\bar t,\bar x}-x|< \eta\}}|\bold p_{\tau^\epsilon-}-\bar p|\big |\mathcal{F}_{\bar t-}\bigg]-2\epsilon(t-\bar t).
\end{array}
\end{equation}
Now by It\^o's formula and since the derivatives of $\phi$ are uniformly bounded we have that 
\begin{eqnarray}
&&\left|\mathbb{E}_{\mathbb{P}^\epsilon}\left[\phi(\tau_\epsilon,X^{\bar t,\bar x}_{\tau^\epsilon},\bar p)-\phi(\bar t,\bar x,\bar p)\big|\mathcal{F}_{\bar t-}\right]\right|\leq c \mathbb{E}_{\mathbb{P}^\epsilon}[ (\tau^\epsilon-\bar t)\ |\mathcal{F}_{\bar t-}]\leq c(t-\bar t).
\end{eqnarray}
Next, let $f:[\bar t,t]\times\mathbb{R}^n\rightarrow \mathbb{R}^n$ be a smooth bounded function, with bounded derivatives.  Recall that under any ${\mathbb{P}}\in\mathcal{P}^f(\bar t, \bar p)$ the process $\bold p$ is strongly orthogonal to $B$. So since under ${\mathbb{P}^\epsilon}$ the process $\bold p$ is a martingale with $\mathbb{E}_{\mathbb{P}^\epsilon}\left[\bold p_{\tau^\epsilon-}|\mathcal{F}_{\bar t-}\right]=\bar p$, we have by It\^o's formula that 
\begin{equation*}
\mathbb{E}_{\mathbb{P}^\epsilon}\left[f_i(\tau^\epsilon,X^{\bar t,\bar x}_{\tau^\epsilon})(\bold p_{\tau^\epsilon-}-\bar p)_i\big|\mathcal{F}_{\bar t-}\right]=\mathbb{E}_{\mathbb{P}^\epsilon}\bigg[\int_{\bar t}^{\tau^\epsilon} \left((\frac{\partial}{\partial t}+\mathcal{L})f_i(s,X^{\bar t,\bar x}_{s})\right)( \bold p_s-\bar p)_i ds\big|\mathcal{F}_{\bar t-}\bigg].
\end{equation*}
Hence by the assumption on the coefficients of the diffusion (A)(i)
\begin{eqnarray}
&&\left|\mathbb{E}_{\mathbb{P}^\epsilon}\left[\langle\frac{\partial \phi }{\partial p}({\tau^\epsilon},X^{\bar t,\bar x}_{\tau^\epsilon},\bar p),\bold p_{{\tau^\epsilon-}}-\bar p\rangle\big|\mathcal{F}_{\bar t-}\right]\right|\leq c \mathbb{E}_{\mathbb{P}^\epsilon}[ (\tau^\epsilon-\bar t)\ |\mathcal{F}_{\bar t-}]\leq c(t-\bar t).
\end{eqnarray}
Furthermore observe that, since $|\bold p_{{\tau^\epsilon-}}-\bar p|\leq1$, we have, that for $\epsilon'>0$ by Young and H\"older inequality 
\begin{equation*}
\begin{array}{l}
\mathbb{E}_{\mathbb{P}^\epsilon}\left[1_{\{|X^{\bar t,\bar x}_{{\tau^\epsilon}}-\bar x|<\eta\}}|\bold p_{{\tau^\epsilon-}}-\bar p|^2\big|\mathcal{F}_{\bar t-}\right]\\
\ \\
\ \ \ \geq \mathbb{E}_{\mathbb{P}^\epsilon}\left[|\bold p_{\tau^\epsilon-}-\bar p|^2\big|\mathcal{F}_{\bar t-}\right]-\frac{1}{\eta} \mathbb{E}_{\mathbb{P}^\epsilon}\left[|X^{\bar t,\bar x}_{\tau^\epsilon}-\bar x| |\bold p_{\tau^\epsilon-}-\bar p|^2\big|\mathcal{F}_{\bar t-}\right]\\
\ \\
\ \ \ \geq (1- \frac{\epsilon' }{\eta} )\mathbb{E}_{\mathbb{P}^\epsilon}\left[|\bold p_{\tau^\epsilon-}-\bar p|^2\big|\mathcal{F}_{\bar t-}\right]-\frac{1}{4\eta\epsilon'} \mathbb{E}_{\mathbb{P}^\epsilon}\left[|X^{\bar t,\bar x}_{\tau^\epsilon}-\bar x|^2\big|\mathcal{F}_{\bar t-}\right]\\
\ \\
\ \ \ \geq (1- \frac{\epsilon' }{\eta} )\mathbb{E}_{\mathbb{P}^\epsilon}\left[|\bold p_{\tau^\epsilon-}-\bar p|^2\big|\mathcal{F}_{\bar t-}\right]-\frac{c}{4\eta\epsilon'}  \mathbb{E}_{\mathbb{P}^\epsilon}[ (\tau^\epsilon-\bar t)\ |\mathcal{F}_{\bar t-}],
\end{array}
\end{equation*}
hence 
\begin{equation}
\begin{array}{l}
\mathbb{E}_{\mathbb{P}^\epsilon}\left[1_{\{|X^{\bar t,\bar x}_{\tau^\epsilon}-\bar x|<\eta\}}|\bold p_{\tau^\epsilon-}-\bar p|^2\big|\mathcal{F}_{\bar t-}\right]\\
\ \\
\ \ \ \geq(1- \frac{\epsilon' }{\eta} )\mathbb{E}_{\mathbb{P}^\epsilon}\left[|\bold p_{\tau^\epsilon-}-\bar p|^2\big|\mathcal{F}_{\bar t-}\right]-\frac{c}{4\eta\epsilon'} (t-\bar t).
\end{array}
\end{equation}
Choosing $0<\epsilon'<\eta$ and combining (8.11) with the estimates (8.12)-(8.14) there exists a constant $c$, such that
\begin{equation}
\mathbb{E}_{\mathbb{P}^\epsilon}\left[|\bold p_{\tau^\epsilon-}-\bar p|^2\big|\mathcal{F}_{\bar t-}\right]\leq c (t-\bar t).
\end{equation}
This implies in particular for $h>0$ by Doob's inequality
\begin{equation}
\mathbb{P}^{\epsilon}\left[\sup_{s\in[\bar t,\tau^\epsilon[}|\bold p_{s}-\bar p|>h\right]\leq c \frac{ \mathbb{E}_{\mathbb{P}^\epsilon}\left[|\bold p_{\tau^\epsilon-}-\bar p|^2\right]}{h^2}\leq c\frac{ (t-\bar t)}{h^2}.
\end{equation}
\ \\
\emph{Step 2: Viscosity supersolution property}\\
To show the viscosity supersolution property we have to show that 
\begin{equation*}
W^-(\bar t,\bar x,\bar p) -\langle h(\bar t,\bar x),\bar p\rangle=\phi(\bar t,\bar x,\bar p) -\langle h(\bar t,\bar x),\bar p\rangle<0
\end{equation*}
implies
\begin{equation*}
(\frac{\partial \phi} {\partial t}+\mathcal{L})[\phi](\bar{t},\bar x,\bar p)\leq0.
\end{equation*}
We will argue by contradiction. Assume that
\begin{equation}
\phi(\bar t,\bar x,\bar p) -\langle h(\bar t,\bar x),\bar p\rangle<0 \ \ \ \ \ \ \ \ \ \ \ \ \ \textnormal{and}\ \ \ \ \ \ \ \ \ \ \ \  (\frac{\partial \phi} {\partial t}+\mathcal{L})[\phi](\bar{t},\bar x,\bar p)>0.
\end{equation}
Then there exist $h,\delta>0$ such that for all $(s,x,p)\in[\bar t,\bar t+h]\times B(\bar x,\bar p)$
\begin{equation}
\begin{array}{l}
\langle h(s,x),\bar p\rangle-\phi(s,x,\bar p) \geq \delta\ \ \ \ \ \ \ \ \ \ \ \ \ \textnormal{and}\ \ \ \ \ \ \ \ \ \ \ \ 
(\frac{\partial \phi} {\partial t}+\mathcal{L})[\phi](s,x,\bar p)\geq \delta.
\end{array}
\end{equation}
By the It\^o formula we have, that 
\begin{equation}
\begin{array}{rcl}
\phi(\tau^\epsilon,X^{\bar t,\bar x}_{\tau^\epsilon},\bold p_{\tau^\epsilon})\geq\phi(\bar t,\bar x,\bar p)+\int_{\bar t}^{\tau^\epsilon} (\frac{\partial}{\partial t}+\mathcal{L})[\phi] (s,X^{\bar t,\bar x}_s,\bold p_s) ds,
\end{array}
\end{equation}
where we used the fact that by the convexity of $\phi$ we have ${\mathbb{P}^\epsilon}$-a.s., that
\begin{eqnarray*}
\sum_{ \bar t\leq r< {\tau^\epsilon}} \left(\phi(r,X^{\bar t,\bar x}_r,\bold p_{r})-\phi(r,X^{\bar t,\bar x}_r,\bold p_{r-})-\langle \frac{\partial}{\partial p}\phi(r,X^{\bar t,\bar x}_r,\bold p_{r-}),\bold p_{r}-\bold p_{r-}\rangle\right) \geq0.
\end{eqnarray*}
Define $A:=\{\inf_{s\in [\bar t, t]} |\bold p_{s-}-\bar p |>h\}$ and $B:=\{\inf_{s\in[\bar t,t]} |X^{\bar t,\bar x}_t-\bar x |>h\}$.\\
Note that by (8.16) and since $\mathbb{P}^\epsilon[B]\leq \frac{c(t-\bar t)^2}{h^4}$ we have that 
\begin{equation}
\begin{array}{rcl}
\mathbb{E}_{\mathbb{P}^\epsilon}[1_A1_B]&\leq& c(\mathbb{E}_{\mathbb{P}^\epsilon}[1_A])^\frac{1}{2}(\mathbb{E}_{\mathbb{P}^\epsilon}[1_B])^\frac{1}{2}\\
\ \\
& \leq& c(\frac{ (t-\bar t)}{h^2})^\frac{1}{2}(\frac{(t-\bar t)^2}{h^4})^\frac{1}{2}=c \frac{(t-\bar t)^\frac{3}{2}}{h^3}.
\end{array}
\end{equation}
Now we can continue as in the proof of Theorem 5.1. By using (8.19) we get
\begin{equation*}
\begin{array}{rcl}
\phi(\bar t,\bar x,\bar p)&\leq&\mathbb{E}_{\mathbb{P}^\epsilon}\left[1_{A^c}1_{B^c}\left(\phi(\tau^\epsilon,X^{\bar t,\bar x}_{\tau^\epsilon},\bold p_{\tau^\epsilon-})-\int_{\bar t}^{\tau^\epsilon} (\frac{\partial}{\partial t}+\mathcal{L})[\phi] (s,X^{\bar t,\bar x}_s,\bold p_s) ds\right)\right]\\
\ \\
&&\ \ \ \ \ \ \ \ \ \ \ \ +c \frac{(t-\bar t)^\frac{3}{2}}{h^3}\\
\ \\
&\leq&\mathbb{E}_{\mathbb{P}^\epsilon}\left[ \langle h(\tau^\epsilon,X^{\bar t,\bar x}_{\tau^\epsilon}),\bold p_{\tau^\epsilon}\rangle 1_{\tau^\epsilon<t}+  \phi(t ,X^{\bar t,\bar x}_{t},\bold p_{t-})1_{\tau^\epsilon=t}\right]\\
\ \\
&&\ \ \ \ \ \ \ \ \ \ \ \ -\delta \mathbb{E}_{\mathbb{P}^\epsilon}\left[ 1_{\tau^\epsilon<t}\right]-\delta \mathbb{E}_{\mathbb{P}^\epsilon}\left[(\tau^\epsilon-\bar t)\right]+2c \frac{(t-\bar t)^\frac{3}{2}}{h^3}.
\end{array}
\end{equation*}
As in (5.7) we have that for $1\geq(t-\bar t)$
\begin{equation}
\begin{array}{l}
(t-\bar t)\leq \mathbb{E}_{\mathbb{P}^\epsilon}\left[ 1_{\tau^\epsilon<t}\right]+\mathbb{E}_{\mathbb{P}^\epsilon}\left[(\tau^\epsilon-\bar t)\big|\mathcal{F}_{\bar t-}\right]
\end{array}
\end{equation}
so
\begin{equation*}
\begin{array}{rcl}
\phi(\bar t,\bar x,\bar p)&\leq&\mathbb{E}_{\mathbb{P}^\epsilon}\left[ \langle h(\tau^\epsilon,X^{\bar t,\bar x}_{\tau^\epsilon}),\bold p_{\tau^\epsilon}\rangle 1_{\tau^\epsilon<t}+  \phi(t ,X^{\bar t,\bar x}_{t},\bold p_{t-})1_{\tau^\epsilon=t}\right]\\
\ \\
&&\ \ \ \ \ -\delta(t-\bar t)+2c \frac{(t-\bar t)^\frac{3}{2}}{h^3},
\end{array}
\end{equation*}
which gives with (8.19)
\begin{equation*}
-\delta(t-\bar t)+2c \frac{(t-\bar t)^\frac{3}{2}}{h^3}+2\epsilon(t-\bar t)\geq0.
\end{equation*}
Dividing by $(t-\bar t)$ we have
\begin{equation}
-\delta+2c \frac{(t-\bar t)^\frac{1}{2}}{h^3}+2\epsilon\geq0.
\end{equation}
However (8.22) contradicts $\delta>0$, since $\epsilon$ and $t-\bar t$ can be chosen arbitrarily small.
\end{proof}

The proof of Theorem 6.4 is now straightforward using the subsolution property of $W^+$, the supersolution property of $W^-$ and the comparison result of Theorem 3.7.


\section{Appendix: Comparison}

In this section we provide the proof of the comparison result Theorem 3.7. for the fully non linear variational PDE (3.5)
\begin{equation*}
				\begin{array}{l}
	\max \bigg\{ \max\{\min\{(-\frac{\partial} {\partial t}-\mathcal{L})[w],w-\langle f(t,x),p\rangle\},\\
	 \ \ \ \ \ \ \ \ \ \ \ \ \ \ \ \ \ \ \ \ \ \ \ \ \ \ \ \ \ \ \ \ \ \ \ \ \ \ \ \ \ \ \ \ \ w-\langle h(t,x),p\rangle\},-\lambda_{\min}\left(p,\frac{\partial ^2 w}{\partial p^2}\right)\bigg\}=0
		\end{array}
			\end{equation*}
with terminal condition $w(T,x,p)=\sum_{i=1,\ldots,I}p_ig_i(x)$. The proof is more or less a straight forward adaption of the results in \cite{Ca}.

\subsection{Reduction to the faces}

Let $\tilde I\subset\{1,\ldots, I\}$ and we define the set $\Delta(\tilde I)$ by
\begin{equation}
 \Delta(\tilde I)=\{p\in\Delta(I): p_i=0 \textnormal{ if } i\not\in\tilde I\}.
\end{equation}

Note that by Definition 3.3. the supersolution property is obviously preserved under restriction. We just state
\begin{prop}
Let $w:[0,T] \times \mathbb{R}^d\times \Delta(I)\rightarrow\mathbb{R}$  be a bounded, continuous viscosity supersolution to (3.5), which is uniformly Lipschitz continuous in $p$. Then the restriction of $w$ to  $\Delta(\tilde I)$ is a supersolution to (3.5) on $[0,T] \times \mathbb{R}^d\times \Delta(\tilde I)$.
\end{prop}

The subsolution property is however not immediate, since $\textnormal{Int}( \Delta(\tilde I))\not\subseteq\textnormal{Int}( \Delta(I))$.

\begin{prop}
Let $w:[0,T] \times \mathbb{R}^d\times \Delta(I)\rightarrow\mathbb{R}$ be a bounded, continuous viscosity subsolution to (3.5), which is uniformly Lipschitz continuous in $p$. Then the restriction of $w$ to  $\Delta(\tilde I)$ is a subsolution to (3.5) on $[0,T] \times \mathbb{R}^d\times \Delta(\tilde I)$.
\end{prop}

\begin{proof}
Set $\tilde w=w\big |_{\Delta(\tilde I)}$. Let  $(\bar t,\bar x, \bar p)\in(0,T)\times\mathbb{R}^d\times\textnormal{Int}(\Delta(\tilde I))$ and $\phi: [0,T] \times \mathbb{R}^d\times \Delta(\tilde I)\rightarrow \mathbb{R}$ a test function such that $\tilde w- \phi$ has a strict minimum at $(\bar t,\bar x,\tilde p)$ with $\tilde w(\bar t,\bar x,\tilde p)- \phi(\bar t,\bar x,\tilde p)=0$. By using the viscosity subsolution property of $w$ on $[0,T] \times \mathbb{R}^d\times \Delta(I)$ we have to show:
\begin{itemize}
\item [(i)] $\lambda_{\min}\left(\tilde p,\frac{\partial ^2 \phi}{\partial p^2}\right)\geq 0$
\item[(ii)] \begin{equation*}
				\begin{array}{l}
	\max\{\min\{(-\frac{\partial } {\partial t}-\mathcal{L})[\phi],\phi-\langle f(t,x),p\rangle\},\phi-\langle h(t,x),p\rangle\}\leq0
		\end{array}
			\end{equation*}
			at $(\bar t,\bar x, \tilde p)$.
\end{itemize}
 However $\tilde p\not \in \textnormal{Int}( \Delta(I))$ so we have to use an appropriate approximation.  Let $\mu\in\mathbb{R}^I$ such that $\mu_i=0$ if $i\in\tilde I$ and $\mu_i=1$ else. Furthermore we define for $p\in\Delta(I)$ the projection $\Pi$ onto $\Delta(\tilde I)$ by
\begin{equation*}
\Pi(p)_i= 
\begin{cases} p_i+\left(\sum_{j\not\in\tilde I} p_i\right)/|\tilde I| & \text{if $i\in\tilde{I}$,}
\\
0 &\text{else.}
\end{cases}
\end{equation*}
\\
Since $w$ is uniformly Lipschitz continuous with Lipschitz constant $k$ with respect to $p$, we have 
\begin{equation*}
w(t,x,p)\leq \tilde w(t,x,\Pi(p))+ (k+1) |\Pi(p)-p|
\end{equation*}
with an equality for $p\in\Delta(\tilde I)$, hence
\begin{equation*}
w(t,x,p)\leq \phi (t,x,\Pi(p))+ 2(k+1)\langle \mu,p\rangle 
\end{equation*}
with an equality only at $(\bar t, \bar x,\tilde p)$, where we used
\begin{equation*}
 |\Pi(p)-p|\leq\sum_{j\in\tilde I} |\Pi(p)_j-p_j|+\sum_{j\not\in\tilde I} p_j=(1+1/|\tilde I|)\sum_{j\not\in\tilde I} p_j=2\langle \mu,p\rangle.
\end{equation*}
For $\epsilon>0$ small we now consider 
\begin{equation}
\max_{(t,x,p)\in[0,T]\times\mathbb{R}^d\times\Delta (I)} w(t,x,p)-\phi_\epsilon(t,x,p)
\end{equation}
 with
 \[\phi_\epsilon(t,x,p)=\phi(t,x,\Pi (p))+ 2(k+1)\langle \mu,p\rangle -\epsilon \sigma (p)\]
and  $\sigma(p)=\sum_{j\not \in\tilde I} \ln (p_i(1-p_i))$. For $\epsilon$ sufficiently small this problem has a maximum $(t_\epsilon,x_\epsilon,p_\epsilon)$ which converges to $(\bar t,\bar x,\tilde p)$ as $\epsilon\downarrow 0$. By the definition of $\sigma$ and the fact that $\tilde p\in\text{Int}(\Delta(\tilde I))$ we have that $p_\epsilon\in\text{Int}(\Delta( I))$. Hence by the subsolution property of $w$ we have, that $\lambda_{\min}\left(p_\epsilon,\frac{\partial ^2 \phi_\epsilon}{\partial p^2}\right)(t_\epsilon, x_\epsilon,p_\epsilon)\geq 0$.
Note that since $\Pi$ is affine, $\Pi|_{\Delta(\tilde I)}=\text{id}$ and $\sigma$ does not depend on $p_i$ for $i\in\tilde I$, we have
\begin{equation*}
\begin{array}{l}
\liminf_{\epsilon\downarrow 0} \lambda_{\min}\left(p_\epsilon,\frac{\partial ^2 \phi_\epsilon}{\partial p^2}\right)(t_\epsilon, x_\epsilon,p_\epsilon)\\
\ \\
\ \ \ \leq \liminf_{\epsilon\downarrow 0} \min_{z\in T_{\Delta(\tilde I)(\tilde p)}\setminus\{0\}} \frac{\langle {\partial ^2 \phi_\epsilon(t_\epsilon, x_\epsilon,p_\epsilon)}z,z\rangle}{|z|^2}\\
\ \\
\ \ \ \leq \liminf_{\epsilon\downarrow 0} \min_{z\in T_{\Delta(\tilde I)(\tilde p)}\setminus\{0\}} \frac{\langle {\partial ^2 \phi_\epsilon(t_\epsilon, x_\epsilon,\Pi(p_\epsilon))}z,z\rangle}{|z|^2}
=\lambda_{\min}\left(\tilde p,\frac{\partial ^2 \phi}{\partial p^2}\right)(\bar t,\bar  x,\tilde p).
\end{array}
\end{equation*}
And since $\lambda_{\min}\left(p_\epsilon,\frac{\partial ^2 \phi_\epsilon}{\partial p^2}\right)(t_\epsilon, x_\epsilon,p_\epsilon)\geq 0$, we have
\begin{equation}
\lambda_{\min}\left(\tilde p,\frac{\partial ^2 \phi}{\partial p^2}\right)(\bar t,\bar  x,\tilde p)\geq 0.
\end{equation}
(ii) follows then by the subsolution property of $w$, i.e. 
\begin{equation}
\begin{array}{l}
		\max\big\{\min\big \{(-\frac{\partial } {\partial t}-\mathcal{L})[\phi_\epsilon](t_\epsilon, x_\epsilon,p_\epsilon)),\\
		\ \\
		\ \ \ \ \ \ \ \ \ \ \ \ \ \ \phi(t_\epsilon, x_\epsilon,p_\epsilon)-\langle f(t_\epsilon, x_\epsilon), p_\epsilon \rangle\big\}
		,\phi(t_\epsilon, x_\epsilon,p_\epsilon)-\langle h(t_\epsilon, x_\epsilon), p_\epsilon \rangle\big\} \leq 0
\end{array}
\end{equation}
by letting $\epsilon\downarrow0$.
\end{proof}

\subsection{Proof of Theorem 3.7}

Let $w_1:[0,T] \times \mathbb{R}^d\times \Delta(I)\rightarrow\mathbb{R}$ be a bounded, continuous viscosity subsolution to (3.5), which is uniformly Lipschitz continuous in $p$, and $w_2:[0,T] \times \mathbb{R}^d\times \Delta(I)\rightarrow\mathbb{R}$ be a bounded, continuous  viscosity supersolution to (3.5), which is uniformly Lipschitz continuous in $p$. Assume that 
\begin{equation}
w_1(T,x,p)\leq w_2(T,x,p)
\end{equation}
for all $x\in\mathbb{R}^d,p\in\Delta(I)$. We want to show that
\begin{equation}
w_1(t,x,p)\leq w_2(t,x,p)
\end{equation}
for all $(t,x,p)\in[0,T]\times\mathbb{R}^d\times\Delta(I)$. As in \cite{Ca} we prove (9.6) by induction over $I$. Indeed if $I=1$, (3.5) reduces to
\begin{equation}
\begin{array}{l}
		\max\big\{\min\big \{(-\frac{\partial } {\partial t}-\mathcal{L})[w],w- f_1(t, x)\big\},w- h_1(t, x)\big\}=0,
\end{array}
\end{equation}
where comparison is a classical result, see e.g. \cite{HH}.
Assume that Theorem 3.7. holds for $I\in\mathbb{N}^*$. That means for $w_1,w_2: [0,T]\times\mathbb{R}^d\times\Delta(I+1)$ we have by Proposition 9.1. and 9.2. that
\begin{equation}
	w_1(t,x,p)\leq w_2(t,x,p) \ \ \ \ \ \forall (t,x,p)\in[0,T]\times\mathbb{R}^d\times\partial({\Delta(I)}).
\end{equation}
We will show (9.6) by contradiction. Assume
\begin{equation}
M:=\sup_{(t,x,p)\in[0,T]\times\mathbb{R}^d\times\Delta(I))} (w_1-w_2)>0.
\end{equation}
Since $w_1$ and $w_2$ are bounded we have for $\epsilon,\alpha,\eta>0$ that
\begin{equation}
\begin{array}{l}
M_{\epsilon,\alpha,\eta}:= \max_{(t,s,x,y,p)\in[0,T]^2\times\mathbb{R}^{2d}\times\Delta(I)}\bigg\{ w_1(t,x,p)-w_2(s,y,p)
\ \\
\ \ \ \ \ \ \ \ \ \ \ \ \ \ \ \ \ \ \ \ \ \ \ \ \ \ \ \ \ \ \ \ \ \ \ \ \ \ \ \ \ \ \ \ \ \ \ \ \ -\frac{|t-s|^2+|x-y|^2}{2\epsilon}-\frac{\alpha}{2} (|x|^2+|y|^2)+\eta t\bigg\}
\end{array}
\end{equation}
is finite and achieved at a point $(\bar t,\bar s,\bar x,\bar y,\bar p)$ (dependent on ${\epsilon,\alpha,\eta}$). Furthermore we have for the limit
\begin{equation}
\lim_{\epsilon,\alpha,\eta\downarrow 0} M_{\epsilon,\alpha,\eta}=\sup_{(t,x,p)\in[0,T]\times\mathbb{R}^d\times\Delta(I))} (w_1-w_2)=M>0.
\end{equation}
With (9.5) and the H\"older continuity of $w_1$ and $w_2$ we have with (9.11) that $\bar t,\bar s<T$ for $\epsilon,\alpha,\eta$ small enough. Also note that $\bar p\in\text{Int}(\Delta(I))$ as soon as $M_{\epsilon,\alpha,\eta}>0$.\\
We now consider a new penalization: For $\beta,\delta>0$ small
\begin{equation}
\begin{array}{l}
M_{\epsilon,\alpha,\eta,\delta,\beta}:=\max_{(t,s,x,y,p,q)\in[0,T]^2\times\mathbb{R}^{2d}\times\Delta(I)^2}\bigg\{ w_1(t,x,p)-w_2(s,y,p)\\
\ \\
\ \ \ \ \ \ \ \ \ \ \ \ \ \ \ \ \ \ \ \ \ \ \ \ -\frac{|t-s|^2+|x-y|^2}{2\epsilon}-\frac{|p-q|}{2\delta}-\frac{\alpha}{2} (|x|^2+|y|^2)+\eta t+\frac{\beta}{2}(|p|^2+|q|^2)\bigg\}
\end{array}
\end{equation}
is attained at a point  $(\tilde t,\tilde s,\tilde x,\tilde y,\tilde p,\tilde q)$ (dependent on ${\epsilon,\alpha,\eta,\delta,\beta}$), where
\begin{equation}
\frac{|\tilde t-\tilde s|^2+|\tilde x-\tilde y|^2}{2\epsilon},\frac{|\tilde p-\tilde q|}{2\delta},{\alpha}|\tilde x|^2,\alpha |\tilde y|^2, \beta|\tilde p|^2,\beta |\tilde q|^2\leq 2(|w_1|_\infty+|w_2|_\infty).
\end{equation}
Furthermore we have with (9.11)
\begin{equation}
w_1(\tilde t,\tilde x,\tilde p)-w_2(\tilde s,\tilde y,\tilde q)>0.
\end{equation}
So for $\beta,\delta\downarrow0$ $(\tilde t,\tilde s,\tilde x,\tilde y,\tilde p,\tilde q)$ converges (up to subsequences) to some $(\bar t,\bar s,\bar x,\bar y,\bar p,\bar p)$, where $(\bar t,\bar s,\bar x,\bar y,\bar p)$ is a maximum point of (9.12). Hence for $\beta,\delta$ sufficiently small we have that $\tilde p,\tilde q\in\text{Int}({\Delta(I)})$.

From the usual maximum principle (see e.g. \cite{CIL}) we have that:\\
for all $\sigma\in(0,1)$ there exist $X_1,X_2\in S^d$, $P_1,P_2\in S^I$ such that on $[0,T]^2\times\mathbb{R}^{2d}\times T_I$ with $ T_I=\{z\in\mathbb{R}^I: \sum_i z_i=0\}$ we have
\begin{equation*}
\left(\frac{\tilde t-\tilde s}{\epsilon}-\eta,\ \frac{\tilde x-\tilde y}{\epsilon}+\alpha\tilde x,\ \frac{\tilde p-\tilde q}{\delta}-\beta \tilde p,\  X_1,\ P_1|_{T_I})\right)\in\bar{\mathcal{D}}^{1,2,2,-}w_1(\tilde t,\tilde x,\tilde p)
\end{equation*}
and
\begin{equation*}
\left (\frac{\tilde t-\tilde s}{\epsilon},\ \frac{\tilde x-\tilde y}{\epsilon}-\alpha\tilde y,\ \frac{\tilde p-\tilde q}{\delta}+\beta \tilde q,\  X_2,\ P_2|_{T_I}\right)\in\bar{\mathcal{D}}^{1,2,2,+}w_2(\tilde s,\tilde y,\tilde q)
\end{equation*}
with
\begin{equation*}
\text{diag} \left( \left(
\begin{array}{cc}
X_1 & 0  \\
0 & -X_2 
\end{array}\right),
\left(
\begin{array}{cc}
P_1|_{T_I} & 0  \\
0 & -P_2|_{T_I}  \end{array}\right)\right)\leq A+\sigma A^2,
\end{equation*}
where
\begin{equation*}
A=\text{diag}
\left\{ \frac{1}{\epsilon}\left(
\begin{array}{cc}
\text{id}_d & -\text{id}_d   \\
-\text{id}_d  & \text{id}_d  
\end{array}\right)+\alpha \text{id}_{2d} ,
\frac{1}{\delta}\left(
\begin{array}{cc}
\text{id}_I  & -\text{id}_I   \\
-\text{id}_I  & \text{id}_I  \end{array}\right)-\beta\text{id}_{2I} \right\}.
\end{equation*}
Note that 
\begin{equation}
\left(
\begin{array}{cc}
X_1 & 0  \\
0 & -X_2 
\end{array}\right)\leq \left(\frac{1}{\epsilon}+2\frac{\sigma}{\epsilon^2}+2\frac{\alpha\sigma}{\epsilon}\right)\left(
\begin{array}{cc}
\text{id}_d & -\text{id}_d   \\
-\text{id}_d  & \text{id}_d  
\end{array}\right)+(\alpha+\alpha^2\sigma) \text{id}_{2d}
\end{equation}
and
\begin{equation}
(P_1-P_2)|_{T_I}\leq  (-\beta+\sigma \beta^2)  \text{id}_{2I}.
\end{equation}
Since $w_1$ is a viscosity subsolution to (3.5) we have
\begin{equation}
\lambda_{\min}(\tilde p,P_1)\geq 0.
\end{equation}
And since $\tilde p\in\text{Int}(\Delta(I))$, this yields with (9.16) to
\begin{equation}
\lambda_{\min}(\tilde q,P_2)>0.
\end{equation}
Furthermore since $w_1$ is a viscosity subsolution and $w_2$ is a viscosity supersolution we have
\begin{equation}
\begin{array}{rcl}
w_1(\tilde t,\tilde x,\tilde p)&\leq& \langle h(\tilde t,\tilde x),\tilde p\rangle \\
\ \\
w_2(\tilde s,\tilde y,\tilde q)&\geq& \langle f(\tilde s,\tilde y),\tilde q\rangle,
\end{array}
\end{equation}
which yields for $\epsilon,\alpha,\eta,\delta,\beta$ small enough with (9.11)
\begin{equation}
\begin{array}{rcl}
w_1(\tilde t,\tilde x,\tilde p)&>& \langle f(\tilde t,\tilde x),\tilde p\rangle \\
\ \\
w_2(\tilde s,\tilde y,\tilde q)&<& \langle h(\tilde s,\tilde y),\tilde q\rangle.
\end{array}
\end{equation}
So again using the subsolution property of $w_1$ and the supersolution property of $w_2$ we have with (9.20)
\begin{equation}
\begin{array}{rcl}
\frac{\tilde t-\tilde s}{\epsilon}-\eta + \frac{1}{2}\textnormal{tr}(aa^*(\tilde t,\tilde x)X_1)+b(\tilde t,\tilde x)\left(\frac{\tilde x-\tilde y}{\epsilon}+\alpha\tilde x\right)&\geq& 0 \\
\ \\
\frac{\tilde t-\tilde s}{\epsilon}+\frac{1}{2}\textnormal{tr}(aa^*(\tilde s,\tilde y)X_2)+b(\tilde t,\tilde x)\left(\frac{\tilde x-\tilde y}{\epsilon}-\alpha\tilde y\right)&\leq& 0.
\end{array}
\end{equation}
Now using (9.15) and (9.16) in (9.21) yields a contradiction for $\epsilon,\alpha,\eta$ sufficiently small as in the standard case (see \cite{CIL}).

\ \\

\bibliographystyle{plain}
\bibliography{bib}



\end{document}